\documentclass{amsart}

\usepackage{amsmath}
\usepackage{amssymb}
\usepackage{amsthm}
\usepackage[mathscr]{eucal}
\usepackage{mathrsfs}
\usepackage{psfrag}
\usepackage{fullpage}
\usepackage{color}
\usepackage{eucal}
\usepackage[dvips]{graphicx}
\usepackage{amsfonts}%
\usepackage{amsaddr}
\theoremstyle{plain}
\newtheorem{theorem}{Theorem}[section]
\newtheorem{definition}{Definition}[section]

\newtheorem{corollary}[theorem]{Corollary}
\newtheorem{lemma}[theorem]{Lemma}
\newtheorem{assumption}[theorem]{Assumption}
\newtheorem{example}[theorem]{Example}
\newtheorem{remark}[theorem]{Remark}

\newcommand{\lb}{\left\{}
\newcommand{\rb}{\right\}}
\newcommand{\Def}{\overset{\text{def}}{=}}

\newcommand{\R}{\mathbb{R}}
\newcommand{\N}{\mathbb{N}}

\newcommand{\KK}{K}

\newcommand{\eps}{\varepsilon}

\newcommand{\Borel}{\mathscr{B}}
\newcommand{\Pspace}{\mathscr{P}}
\newcommand{\BP}{\mathbb{P}}
\newcommand{\BE}{\mathbb{E}}
\newcommand{\filt}{\mathscr{F}}

\newcommand{\la}{\left \langle}
\newcommand{\ra}{\right\rangle}

\newcommand{\ee}{\mathfrak{e}}
\newcommand{\genA}{\mathcal{A}}
\newcommand{\genL}{\mathcal{L}}

\newcommand{\pp}{\mathsf{p}}
\newcommand{\PP}{\mathcal{P}}
\newcommand{\pt}{\star}
\newcommand{\SSS}{\mathcal{S}}

\newcommand{\Types}{\mathcal{P}}

\newcommand{\NN}{{N,n}}
\newcommand{\mart}{\mathcal{M}}

\newcommand{\QQ}{\mathcal{Q}}
\newcommand{\subprob}{\mathcal{M}_1}

\newcommand{\TInt}{\mathcal{T}}
\newcommand{\PT}{\mathsf{X}}
\newcommand{\ind}{\text{ind}}
\newcommand{\KCal}{\mathcal{K}}
\newcommand{\bnu}{\boldsymbol{\nu}}
\newcommand{\bmu}{\boldsymbol{\mu}}

\allowdisplaybreaks

\begin{document}
\title{Default Clustering in Large Pools: Large Deviations}

\author{Konstantinos Spiliopoulos}
\address{Department of Mathematics \& Statistics\\
Boston University\\
Boston, MA 02215}
\email{kspiliop@math.bu.edu}

\author{Richard B. Sowers}
\address{Department of Mathematics\\
    University of Illinois at Urbana--Champaign\\
    Urbana, IL 61801}
\email{r-sowers@illinois.edu}

\date{\today.}

\begin{abstract}
We study large deviations and rare default clustering events in a dynamic large heterogeneous portfolio of interconnected components.
Defaults come as Poisson events and the default intensities of the different components in the system interact through the empirical default rate and via systematic effects that are common to all components.   We establish the large deviations principle for the empirical  default rate for such an interacting particle system. The rate function is derived in an explicit form that is amenable to numerical computations and derivation of the most likely path to failure for the system itself. Numerical studies illustrate the theoretical findings. An understanding of the role of the preferred paths to large default rates and the most likely ways in which contagion and systematic risk combine to lead to large default rates would give useful insights into how to optimally safeguard against such events.
%
%
\end{abstract}



\thanks{The authors would like to thank Kay Giesecke for  helpful comments.}




\maketitle

\section{Introduction}
The financial crisis of 2007-2008 challenged the mathematical finance
community to understand \emph{connectedness} in financial systems.
Appropriate models need to be developed to understand how risk can propagate
between financial objects which might have heretofore been modeled as closed systems.

It is possible that initial shocks, such as changes in interest rate values, changes of commodities prices or reduction in global economic growth, could trigger contagion effects, e.g., \cite{Meinerding2012}. It is likely then that some transmission mechanism, such as financial linkages or simply investor irrationality, could cause components of the system to be affected by the initial shock. Reduce-form point process models of correlated default, that are usually based on counting processes, are often used to access portfolio credit risk in portfolios of defaultable assets such as loans and corporate defaults. In these models defaults arrive at intensities that are governed by a given system of stochastic differential equations.  Due to the size of the portfolios computing the distribution of the loss from default in these models is often challenging. Main US banks for example may easily have $20,000$ wholesale loans and $50,000-100,000$ mid-market and commercial loans. Mortgage pools with size of $10,000$ are often common. Simulation and analysis of such pools is non-trivial and often quite burdensome.

In this work we focus on using dynamic portfolio credit risk models to study rare events in large portfolios and default clustering. We statistically model failure via the classical framework of point processes. However, we include several distinct and meaningful sources of randomness
 that can lead to failure of the system itself. The intensities form an interacting particle system, where the interaction happens through feedback terms
and exposure to common systematic risks.  The foundation of the empirically motivated model that we study lies in \cite{GSS2013}.
In this work various properties of such models were investigated, finding that contagion and systematic risk give meaningful insights into clustering of defaults.

Typical (law of large numbers) and central limit type behavior for such models as the number of components in the system grows have been studied in  \cite{GSS2013, GSSS2013, SSG2013}. Our interest here is tail behavior and rare events.  One of the principle theoretical
attractions of structured finance was that, supposedly, large pools
of assets were regular enough that one could accurately price tail events.
In hindsight, however, interaction in large pools proved an Achilles' heel. Interconnections often make a system robust, but they can also act as
conduits for failure. In this paper, we seek to understand the mathematics of rare systemic collapse in large interacting systems.

We consider a large system with interacting components that is influenced by an exogenous source of randomness.  There is a central source of interconnection (a central ``bus'').  Failure of any component stresses the central bus, which in turn can cause other components to fail (a feedback effect).  In  particular, we want to understand  how the system can catastrophically fail, and how the feedback mechanism and the exogenous factor interact to produce large failure clusters.
Along with an understanding of the likelihood of such a failure, we want to understand
the structure of pathways to failure.  Given a statistical description of the
interacting system, an understanding of the ``most likely'' pathways to failure would naturally lead
to an understanding of how to control the system to minimize failure and how to sense the onset of failure.

Of particular interest is \emph{dynamic} interactions.   The extra
dimension of time allows one to compare different possible ways
in which a system can end up at a given point.  This may help in early-stage
identification of certain phenomena.  Our goal here is to develop some of these
insights in the model of \cite{GSS2013}.  In particular, we would like to understand
\emph{pathways to systemic collapse}.  In the presence of contagion, what
are the most likely ways that a pool can suffer large losses?  As in other engineered
systems, characterization of most likely paths to failure can better allow
effective intervention.

The heart of our analysis is the theory of \emph{large deviations}, e.g., see the classical manuscripts \cite{DemboZeitouni, FWBook}.  The theory
of large deviations gives a rich framework in which to identify and then prove
rates of decay of exponential tails.  Two simple examples are well-known;
Sanov's theorem, which in its simplest captures the tail behavior of a large
collection of i.i.d coin flips, and the Freidlin-Wentzell theorem, which identifies the
most likely way that diffusive perturbations can drive a stable differential
equation out of equilibrium.  Both of these core examples play a role
in our analysis, even though the situation is more complex here because the coin flips are neither identically distributed, nor independent.  Defaults can be modeled as coin flips, and the effects
of contagion can be thought of as random perturbations of a dynamical system.
Our main theorems, Theorems \ref{T:LDPHeterogeneous1} and  \ref{T:LDPHeterogeneous2}, give the large deviations principle. In the full case, where both contagion and systematic effects are present, the rate function is
 a combination of a relative entropy (a  ``nonlinear''  Sanov's type theorem) and the Freidlin-Wentzell action functional.

Tail behavior in static pools is investigated  in \cite{DDD04,Glasserman} and in \cite{SS11} the authors study the effect of stochastic recovery on the tail of the loss distribution when the recovery rate depends on the default rate. In \cite{GL05,GKS07} rare event asymptotics for the loss distribution in the Gaussian copula model of portfolio credit risk and related  importance sampling questions are studied.  In a large deviations analysis of a mean field model in \cite{daipra-etal} the authors take the default intensity of a component in the pool to be a deterministic function
 of the percentage pool loss due to defaults, see also \cite{daipra-tolotti}. In \cite{ZhangBlanchetGiseckeGlynn2013}, the authors establish a large time large deviations principle
for an interacting system of affine point processes.  In \cite{PapanicolaouSystemicRisk}, the authors study systematic risk [from endogenous aspect] via a mean field model of interacting agents. Using a model of a two-well potential, agents can move freely from a healthy state to a failed state.  The authors study probabilities of transition from the healthy to the failed state  using large deviations ideas.

In this paper, we consider tail behavior for dynamic heterogeneous pools with stochastic intensity and we are interested in the behavior for large pools of interconnected components. Our contribution is two-fold. Firstly,  we develop a large deviations principle for
a dynamic point process model of correlated default timing that takes into account both effects of contagion and of systematic, exogenous, risks. We study the case that the number of constituent
firms or components in the network grows. Secondly, we numerically explore the large deviations results which can help understand how default clusters occur in such systems.

In the theoretical side,  we see that the rate function governing the tail events is given explicitly as an additive functional of a
Freidlin-Wentzell action functional and a nonlinear relative entropy, where the nonlinearity originates from the fact that defaults are not independent. The dependence structure and the heterogeneity of the environment,
i.e., the fact that defaults are not identically distributed, complicate the mathematical analysis. Nevertheless, we rigorously obtain a fairly explicit form of the large deviations principle which is amenable to numerical investigations. In particular, we derive a representation for the rate
 function which is suitable for numerically computing both the rate function and the most likely path to failure, i.e., the extremals, in both homogeneous and heterogeneous environments.

In the numerical side, the numerical experiments illustrate the effect of systematic (exogenous) risks and contagion on tail events of such systems. For instance, as we shall see in the numerical investigations of Section  \ref{S:Numerics}, if a large default cluster occurs, the systematic risk is most likely to play a large role in the initial phase, but then its importance decreases (and thus the contagion effects become more important). Moreover, the analysis of the extremals related to the large deviations rate function (i.e., the most likely paths to failure) in heterogeneous pools with two or more types can help understand which types in the pool are more vulnerable to default due to the impact of contagion, and as a consequence the pathway to creation of default clusters.

We end this introduction by mentioning that, even though the interacting system that we study is primarily motivated by issues of systemic risk in large financial networks,
its formulation is sufficiently generic to make the analysis and results of broader interest. The physical phenomenon of the failure of an individual component in a network, which could have been caused by an external force,
 increasing the stress to other components of the network, making the system more likely to fail, is of broader interest and applicability.

The paper is organized as follows. In Section \ref{LLN}, we describe the model, the assumption and recall the law of large numbers result proven in \cite{GSS2013}. In Section \ref{S:MainResults} we describe our main results, the large deviations principle for the empirical measure and empirical default rate. Then, in Section \ref{S:Numerics} we perform numerical experiments supporting the theoretical findings. Section \ref{S:LDPheterogeneous} contains the proofs of the large deviations results of Section \ref{S:MainResults}.
We conclude with our conclusions in Section \ref{S:Conclusion}.

\section{Model and law of large numbers}\label{LLN}

We assume that our overall system contains $N$ components or subsystems (where $N$ is large). We assume that $(\Omega,\filt,\BP)$ is an underlying probability triple on which all random variables are defined.

To start, let $\tau^{N}_{n}$ be the stopping time at which the $n$-th component (or particle) in our system fails. A failure time $\tau^{N}_{n}$ has intensity process $\lambda^\NN$, which satisfies
\begin{equation*}\label{E:stochhaz}
 \BP\{\tau^{N}_{n}\in (t,t+\delta]|\filt_t,\, \tau^{N}_{n}>t\} \approx \lambda^\NN_t \delta
\end{equation*}
as $\delta\searrow 0$, where $\filt_t$ is the sigma-algebra generated by the entire system
up to time $t$. That is, the process defined by $\chi_{\{\tau^{N}_{n}\le t\}}-\int^{t}_0\lambda^\NN_s 1_{\{\tau^{N}_{n}>s\}}ds$ is a martingale with respect to $\filt_t$.

We want our model to capture several important phenomena.  Each component will be affected by three sources of randomness, one of which is
unique to the component itself, one that is responsible for contagious effects, and one of which reflects the external environment. We assume that each component has been engineered to be stable.  We also assume, however, the system is subject to cascading (or ``contagious'') behavior; failure of any component is likely to lead to failure of more components. In particular, fixing $N\in\N$ and $n\in\{1,\cdots,N\}$ we consider the model
\begin{equation} \label{E:main}
\begin{aligned}
d\lambda^\NN_t &= -\alpha_\NN (\lambda^\NN_t-\bar \lambda_\NN)dt + \sigma_\NN \sqrt{\lambda^\NN_t}dW^n_t +  \beta^C_\NN dL^N_t+ \eps_{N}\beta^S_\NN \lambda^\NN_t dX_t \qquad t>0\\
\lambda^\NN_0 &= \lambda_{\circ, N,n}\\
dX_t &= b( X_t) dt + \kappa(X_t)dV_t \qquad t>0\\
X_0&= x_\circ \\
L^N_t &= \frac{1}{N}\sum_{n=1}^N \chi_{\{\tau^{N}_{n}\le t\}}, \end{aligned}
\end{equation}
where
\begin{equation}\label{E:tau}
\tau^{N}_{n} \Def \inf\lb t\ge 0: \int_{0}^t \lambda^\NN_s ds\ge \ee_n\rb.
\end{equation}
Here the $W^n$'s are an independent collection of standard Brownian motions,
$V$ is also a standard Brownian motion, and the $\ee_n$'s are standard
exponential random variables; the $W^n$'s, $V$ and the $\ee_n$'s are all independent
of each other.

In the case $\beta^{C}_{\NN}=\beta^{S}_{\NN}=0$ for all $n\in\{1,\cdots,N\}$, one recovers the classical CIR model in credit risk, e.g., \cite{DPS}. The term $\beta^C_\NN dL^N_t$ reflects
the contagious effects and the $X_{t}$ term is the external source of randomness (systematic risk).
Notice that movements in $X$ cause correlated changes in each intensity $\lambda^\NN$, which then provides a channel for default clustering.
Each system's component sensitivity to $X$ is measured by the parameter $\beta^S_\NN\in\R$.  Then,  a default
causes a jump of size $\beta^C_\NN/N$ in the intensity $\lambda^\NN$, where $\beta^C_\NN\in \R_+= [0,\infty)$.
The mean-reversion of $\lambda^\NN$ implies that the impact of a default fades away with time, exponentially with rate $\alpha_\NN\in\R_+$. The linear dependence of $\lambda$ in the $dX$ component guarantees that $\lambda^{\NN}_{t}\geq 0$ for every $N\in\N$, $n\in\{1,\cdots,N\}$ and $t\geq 0$ (Proposition 3.3 in \cite{GSS2013}), and makes the analytic computations easier.
As it is demonstrated in \cite{azizpour-giesecke-schwenkler}, an important  channel for clustering of defaults are self-exciting effects of such types.

Proposition 3.3 in \cite{GSS2013} guarantees that under the assumption of an existence of a unique strong solution for the SDE for $X$
process, the system  \eqref{E:main} has a unique strong solution such that $\lambda^{\NN}_{t}\geq 0$ for every $N\in\N$, $n\in\{1,\cdots,N\}$ and $t\geq 0$.
The structure of the feedback term, i.e., the empirical average $L^{N}_{t}$, is of mean field type, which brings the system (\ref{E:main}) roughly within the class of McKean-Vlasov
models, e.g., \cite{Gartner88}. However, as it is also demonstrated in \cite{GSS2013, GSSS2013}, the structure of (\ref{E:main}) presents several difficulties
that bring the analysis of such systems outside the scope of the standard setup.
In particular, a key structural difference is that the empirical measure is not on the level of the intensities $\lambda^\NN_t$,
but on the level of the default times $\tau^{N}_{n}$, i.e. there is an extra level of randomness. Moreover,
the interacting particle system is heterogeneous and not homogeneous, there is the presence of the additional term $X$, the coefficients of the
intensity dynamics are not bounded and there is a square root degeneracy.

Let's define a `type' space which  describes the $\lambda^\NN$'s.  Define $\Types\Def \R_+^4\times \R^2$ and we denote by $\pp^\NN=(\alpha_\NN,\bar \lambda_\NN,\sigma_\NN,\beta^C_\NN,\beta^S_\NN,\lambda_{\circ,N,n})\in \Types$.  Then $\pp^\NN$
gives the dynamics of $\lambda^\NN$, the response to the loss process $L^N$, and the initial condition.  Let's also define $\PP^\NN_t\Def (\alpha_\NN,\bar \lambda_\NN,\sigma_\NN,\beta^C_\NN,\beta^S_\NN,\lambda^\NN_t)$; this is a $\Types$-valued process
which keeps track of the dynamics of $\lambda^\NN$ and its current value.  Of course, we have that $\PP^\NN_0=\pp^\NN_0$.

For any Polish space $S$, let $\subprob(S)$ be
the collection of subprobability Borel measures on $S$;
i.e., $S$ consists of Borel measures $\nu$ on $S$ such that
$\nu(S)\le 1$.  Then $\subprob(S)$ is itself a Polish space;
\cite{MR88a:60130}.  Define here $E\Def \subprob( \Types)$, and let
\begin{equation*}
\bmu^N_t \Def \frac{1}{N}\sum_{n=1}^N\delta_{ \PP^\NN_t}\chi_{\{\tau^{N}_{n}> t\}};
\end{equation*}
this is the empirical distribution of $\PP^\NN_t$ for those components which are still ``alive''.
We note that $\bmu^N_t$ is a random trajectory in $E$.  Also,
\begin{align*}\label{ln}
L^N_t =1-\bmu^N_t( \Types),
\end{align*}

We shall denote by $\Pspace(\Types)$ the set of probability measures on $\Types$. In \cite{GSS2013, GSSS2013,SSG2013}, the authors have developed law of large numbers and central limit theorem approximations to $\bmu^N_t$ and as a consequence to $L^{N}_{t}$ as well. The goal of this paper is to establish a large deviations principle for the empirical loss $L^N_t$, namely to study its tail behavior. In this paper we mainly concentrate on the case $\lim_{N\rightarrow\infty}\eps_{N}=0$. To focus our investigation, let's first recall the law of large numbers results obtained in \cite{GSS2013}; we can only understand 'rare' events relative to the
'typical' event.

\begin{assumption}\label{A:Bounded} We assume that there is a $\KK_{\ref{A:Bounded}}>0$ such that the $\alpha_\NN$'s, $\lambda_\NN$'s, $\sigma_\NN$'s,
$|\beta^C_\NN|$'s, $|\beta^S_\NN|$'s, and $\lambda_{\circ,N,n}$'s are all bounded by  $\KK_{\ref{A:Bounded}}$ for all $N\in \N$ and $n\in \{1,2,\dots, N\}$. \end{assumption}
\noindent Thus the types are bounded.  In fact we want them to have a macroscopic
distribution.
\begin{assumption}\label{A:regularity}  For $N\in \N$, define
\begin{equation*} U_N\Def \frac{1}{N}\sum_{n=1}^N \delta_{  \pp^\NN}.\end{equation*}
We assume that $U\Def \lim_{N\to \infty}U_N$ exists in $\Pspace( \Types)$, in the sense of weak convergence of probability measures on $\Types$.
 \end{assumption}

Under these assumptions, $\mu\Def \lim_{N\to \infty}\bmu^N$ is a well-defined
measure-valued process.  We can identify $\mu$ via the martingale problem.
For $ \pp=(\hat{\pp},\lambda)$ where $\hat{\pp}=(\alpha,\bar \lambda,\sigma,\beta^C,\beta^S)\in \Types$ and $f\in C^\infty( \Types)$, define the operators
\begin{equation*}\label{E:Operators1}
\begin{aligned} (\genL_1 f)( \pp) &= \frac12 \sigma^{2}\lambda\frac{\partial^2 f}{\partial \lambda^2}( \pp) - \alpha(\lambda-\bar \lambda)\frac{\partial f}{\partial \lambda}( \pp)-\lambda f( \pp)\\
(\genL_2 f)( \pp) &= \beta^C \frac{\partial f}{\partial \lambda}( \pp). \end{aligned}
\end{equation*}

Define also
\begin{equation*} \QQ( \pp) \Def \lambda \end{equation*}
for $ \pp=(\hat{\pp},\lambda)$ where $\hat{\pp}=(\alpha,\bar \lambda,\sigma,\beta^C,\beta^S)\in \Types$.
The generator $\genL_1$ corresponds to the diffusive part of the intensity with killing rate $\lambda$, and $\genL_2$ is the macroscopic effect of contagion on the
surviving intensities at any given time.
For every $f\in C^\infty( \Types)$ and $\mu\in E$, define
\begin{equation*} \la f,\mu\ra_E \Def \int_{ \pp \in  \Types}f( \pp)\mu(d \pp). \end{equation*}

Let $\SSS$ be the collection of elements $\Phi$ in $B(\Pspace( \Types))$ of the form
\begin{equation*}\label{E:form} \Phi(\mu) = \varphi\left(\la f_1,\mu\ra_E,\la f_2,\mu\ra_E\dots \la f_M,\mu\ra_E\right) \end{equation*}
for some $M\in \N$, some $\varphi\in C^\infty(\R^M)$ and some $\{f_m\}_{m=1}^M$
For $\Phi\in \SSS$ of the form \eqref{E:form}, define
\begin{equation*} \label{E:limgen} \begin{aligned} (\genA\Phi)(\mu) \Def \sum_{m=1}^M \frac{\partial \varphi}{\partial x_m}\left(\la f_1,\mu\ra_E,\la f_2,\mu\ra_E\dots \la f_M,\mu\ra_E\right) \lb \la \genL_1f_m,\mu\ra_E + \la \QQ,\mu\ra_E \la \genL_2 f_m,\mu\ra_E \rb. \end{aligned}\end{equation*}

We claim that $\genA$ will be the generator of the limiting martingale problem (see \cite{GSS2013}).
\begin{lemma}\footnote{At this point we would like to remark that there is a typo in the formulation of the particular result in \cite{GSS2013}. In particular, it is mentioned there that $(\genL_2 f)( \pp) =  \frac{\partial f}{\partial \lambda}( \pp)$ and $\QQ( \pp) =  \beta^C\lambda$, where it should have been $(\genL_2 f)( \pp) = \beta^C \frac{\partial f}{\partial \lambda}( \pp)$ and $\QQ( \pp) = \lambda$.}[Weak Convergence]\label{L:wconv}   Let $\lim_{N\rightarrow\infty}\eps_{N}=0$. The sequence $\{\bmu^N\}_N$
is tight in $D_{E}([0,T])$.  Moreover,  for any $\Phi\in \SSS$ and $0\le r_1\le r_2\dots r_J=s<t<T$ and $\{\psi_j\}_{j=1}^J\subset B(E)$, we have that
\begin{align*} \lim_{N\to \infty}\BE\left[\lb \Phi(\bmu^N_t)-\Phi(\bmu^N_s)-\int_{r=s}^t (\genA\Phi)(\bmu^N_r)dr\rb \prod_{j=1}^J \psi_j(\bmu^N_{r_j})\right]&=0 \\
\lim_{N\to \infty}\BE\left[\Phi(\bmu^N_0)\right] &= \Phi(U). \end{align*}
The limit $\bmu\Def \lim_N\bmu^N$ uniquely exists and is deterministic, this limit
being in probability in $C([0,T];E)$. Moreover, letting $L_{t}=1-\bmu_{t}({\Types})$, we get that for every $\delta>0$
\[
\lim_{N\rightarrow\infty}\BP\left\{\sup_{0\leq t\leq T}|L^{N}_{t}-L_{t}|\geq \delta  \right\} =0.
\]
\end{lemma}

Corollary \ref{C:LDPgivingLLNHeter} shows that $L_{t}$ can be expressed as the unique solution to a fixed point equation. As in the
law of large numbers, this gives us a reference trajectory for defining rare events.
For any closed subset $F$ of $C([0,T];[0,1])$ which does not contain $L$,
we have that
\begin{equation} \label{E:LDP}\lim_{N\to \infty}\BP\{ L^N\in F\} = 0. \end{equation}

The theory of large deviations gives us the \emph{rate} of decay of
probabilities like in \eqref{E:LDP}, namely the tail of the distribution of large portfolio losses.  It identifies a function $I:C([0,T];[0,1])\to [0,\infty]$ such that, informally,
\begin{equation*} \BP\{ L^N\approx L_{\text{ref}}\} \overset{N\to \infty}{\asymp}
\exp\left[-NI(L_{\text{ref}})\right] \end{equation*}

While rare events are (by definition) unlikely to happen, large deviations theory gives a rigorous
framework to compare the rarity of different rare events.  It thus helps us understand
the 'most likely' rare event in a given set of rare events.  As the numerical experiments in Section \ref{S:Numerics} indicate, larger sensitivity to contagion and systematic risk leads to fatter tails, which leads to larger likelihood of large losses in the system. Insights of this type can help understand the role of contagion and systematic risk and how they interact to produce atypically large failure rates. This can then provide guidance
in guarding against dangerous rare system behavior.

\section{Problem formulation and main results}\label{S:MainResults}
We here present our main results on large deviations for the empirical loss $L^N_t$. We start with some notation and preliminary computations in Subsection \ref{SS:Prelim}. The first result is Theorem \ref{T:LDPHeterogeneous1} in Subsection \ref{S:Heterogeneous}, where we derive the large deviations principle in the case $\eps_N=0$.
Lemma \ref{L:wconv} with $\eps_N=0$ follows as a result; see Corollary \ref{C:LDPgivingLLNHeter}.  Corollary \ref{C:LDPHeterogeneous2} gives an alternate formulation
of the rate function of Theorem \ref{T:LDPHeterogeneous1}; this representation
is useful for numerical studies. Then in Subsection \ref{SS:HeterogeneousSystematic}, we present, in Theorem \ref{T:LDPHeterogeneous2}, the second main result of this paper, which is the large deviations principle when $\eps_N\neq0$, i.e., when both contagion and systematic effects are present.

The large deviations principle of Theorem \ref{T:LDPHeterogeneous1} is obtained
by first identifying a large deviations principle for the empirical measure of defaults in the heterogeneous pool, $\bnu^N$, (properly defined in \eqref{E:nudef}),
and then using the contraction principle. The large deviations principle for $\bnu^N$ is proved in two steps. First, we derive the large deviations principle in the independent (i.e., when all $\beta^{C}_{\NN}=0$), but heterogeneous case. Varadhan's transfer lemma (recalled as Theorem \ref{T:TransferResult2}) then
implies the LDP for the general case. The large deviations principle of Theorem \ref{T:LDPHeterogeneous2} is obtained via a conditioning argument from the  large deviations principle of Theorem \ref{T:LDPHeterogeneous1} and that of small noise diffusion processes, \cite{FWBook}. In this section we present statement of theorems and the corresponding proofs are in Section \ref{S:LDPheterogeneous}.

Let's recall the concept of large deviations and the associated rate function.
\begin{definition}\label{Def:LDP}
If $S$ is a Polish space and $\BP$ is a probability measure on $(S,\Borel(S))$, we say that a collection $(\xi_n)_{n\in \N}$
of $S$-valued random variables has a large deviations principle with rate function $I:S\to [0,\infty]$ and speed $n$ if
\begin{enumerate}
\item For each $s\ge 0$, the set $\Phi(s) = \lb s\in S: I(s)\le s\rb$
is a compact subset of $S$.
\item For every open $G\subset S$,
\begin{equation*}
\varliminf_{n\nearrow\infty}\frac{1}{n}\ln \BP\lb \xi_n\in G\rb \geq -\inf_{s\in G} I(s)
\end{equation*}
\item For every closed $F\subset S$,
\begin{equation*}
\varlimsup_{n\nearrow\infty}\frac{1}{n}\ln \BP\lb \xi_n\in F\rb\leq -\inf_{s\in F} I(s).
\end{equation*}
\end{enumerate}
\end{definition}

\subsection{Preliminary computations}\label{SS:Prelim}

Let's set up some notation.  Fix  $\pp=(\alpha,\bar \lambda,\sigma,\beta^C,\beta^S,\lambda_{\circ})\in \Types$.
Fix also a time horizon $T>0$.
Let a Polish space $S$ and denote by $C(S;\R)$ and $AC(S;\R)$ to be the collection of continuous and respectively absolutely continuous paths from $S$ to $\R$. For notational convenience we will sometimes write $C(S;\R)$ and $AC(S;\R)$ if no confusion arises. For $\varphi$ and $\psi$ in $AC([0,T];\R)$
and $W^*$ a reference Brownian motion, let $\lambda^{\varphi,\psi}$
be the solution of the SDE
\begin{equation} \label{E:DuffiePanSingletonequation}\begin{aligned}
 \lambda^{\varphi,\psi}_t(\pp) &= \lambda_\circ -\alpha \int_{0}^t(\lambda^{\varphi,\psi}_s(\pp)-\bar \lambda)ds + \sigma \int_{0}^t\sqrt{\lambda^{\varphi,\psi}_s(\pp)}dW^*_s \\
&\qquad +\beta^C  \varphi(t)\chi_{[0,T]}(t) + \beta^S \int_{0}^t\lambda^{\varphi,\psi}_s(\pp) \chi_{[0,T]}(s) d\psi(s) \qquad t>0  \end{aligned}\end{equation}

This will represent the conditional intensity of a ``randomly-selected'' name of type ${\pp}$ in our pool.  Let the random variable $\tau^{\varphi,\psi}$ be such that
\begin{equation}\label{E:fdens}
\BP\lb \tau^{\varphi,\psi}\leq t\rb=\BP\lb \int_{0}^t \lambda^{\varphi,\psi}_s ({\pp})ds\geq\ee\rb = 1-\BE\left[\exp\left[-\int_{0}^t \lambda^{\varphi,\psi}_s ({\pp})ds\right]\right] \end{equation}
for all $t>0$, where $\ee$ is an exponential $(1)$ random variable which is independent of $W^*$.
Define
\begin{equation*}\label{E:fdef}
f_{\varphi,\psi}^{{\pp}}(t) \Def \BE\left[\lambda^{\varphi,\psi}_t({\pp})\exp\left[-\int_{0}^t \lambda^{\varphi,\psi}_s({\pp}) ds\right]\right].
\end{equation*}

Then for $t\in[0,T]$
\begin{align}
\BP\lb \tau^{\varphi,\psi}\leq t\rb &= \int_{0}^t f_{\varphi,\psi}^{ {\pp}}(s)ds= 1- \BE\left[\exp\left[-\int_{0}^t  \lambda^{\varphi,\psi}_s({\pp}) ds\right]\right].
\label{Eq:ProbDistribution}
 \end{align}

Let's now collect together all scenarios where $\tau^{\varphi,\psi}>T$.  Fix an abstract point $\star$ not in $[0,T]$ and define
\begin{equation*} \TInt\Def [0,T]\cup \{\star\} \end{equation*}
(which is a Polish space).
For given trajectories $\varphi$ and $\psi$ in $AC([0,T];\R)$,
define $\mu_{\varphi,\psi}^{{\pp}}\in \Pspace(\TInt)$ as
\begin{equation}\label{E:muDef}  \mu_{\varphi,\psi}^{{\pp}}(A) \Def \int_{t\in A\cap [0,T]}f_{\varphi,\psi}^{{\pp}}(t)dt + \delta_{\pt}(A)\lb 1- \int_{0}^T f_{\varphi,\psi}^{{\pp}}(t)dt\rb  \end{equation}
for all $A\in \Borel(\TInt)$. In other words, $ \mu_{\varphi,\psi}^{{\pp}}$ is the corresponding probability measure on  $\TInt$.

The large deviations principles appearing  in Theorems \ref{T:LDPHeterogeneous1} and \ref{T:LDPHeterogeneous2} are in terms $\mu_{\varphi,\psi}^{{\pp}}$. Moreover, taking advantage of the special structure of the problem, Lemma \ref{L:DensityFcn} shows that $\mu_{\varphi,\psi}^{{\pp}}$ can be computed in closed form.

\subsection{Large deviations principle for the case $\eps_N=0$.}\label{S:Heterogeneous}
In this case the exogenous source of randomness, $X$ plays no role in the computations. Define the probability measures $\nu^{N}\in \Pspace({\Types}\times \TInt)$
\begin{equation} \label{E:nudef}
\bnu^{N}=\frac{1}{N}\sum_{n=1}^{N}\delta_{\pp^{\NN},\tau_{n}}\chi_{\{\tau_n\le T\}} + \frac{1}{N}\sum_{n=1}^{N}\delta_{\pp^{\NN},\star}\chi_{\{\tau_n> T\}}
\end{equation}
This captures the distribution of $\pp^\NN$ (the 'type' of the asset) and $\tau^\NN$ (the
default time).

For notational convenience set
\begin{equation*} \PT\Def \Types \times \TInt\end{equation*}
then $\PT$ is Polish.  We
can of course recover the empirical loss from $\bnu^N$;  $L^{N}_{t}=\bnu^{N}({\Types}\times[0,t])$.

Let's make the following definition.
\begin{definition}\label{Def:StochasticKernel}
Fix $\omega\in \Pspace(\PT)$ and $\upsilon\in \Pspace(\Types)$; we say that $\omega=\xi\otimes\upsilon $, where $\xi$ is a measurable map from $ \Types$ to $\Pspace(S)$ \textup{(}i.e., $\xi$ is a stochastic kernel\textup{)}
if
\[
\omega(A\times B)=\int_{\pp\in A}\xi( \pp)(B)\upsilon(d{\pp}).
\]
for all $A\in \Borel( \Types)$ and $B\in \Borel(\TInt)$.  In this case, we write
\begin{equation*} \xi = \frac{\partial \omega}{\partial \upsilon}. \end{equation*}
\end{definition}

Let us define
\begin{equation}\label{E:ClassicalEntropy}
H(\nu,\mu)=\begin{cases}
\int_{t\in \TInt}\ln\frac{d\nu}{d\mu}(t)\nu(dt), &\text{if $\nu\ll\mu$}\\
\infty, &\text{otherwise}
\end{cases}
\end{equation}
and then
\begin{equation}\label{E:HDef}
\bar{H}(\nu)=\begin{cases}
\int_{{\pp}\in{\Types}}H\left(\frac{\partial \nu}{\partial U}({\pp}),\mu_{\bar{\nu},0}^{{\pp}}\right)U(d{\pp}) & \text{if $\frac{\partial\nu}{\partial U}({\pp})$ exists}\\
\infty &\text{otherwise}
\end{cases}
\end{equation}
and
where $\mu_{\bar{\nu},0}^{{\pp}}$ is defined in \eqref{E:muDef},  $U$ is given by Assumption \ref{A:regularity}, the superscript ${\pp}$ denotes the dependence on the particular element $ \pp=(\alpha,\bar \lambda,\sigma,\beta^C,\beta^S,\lambda)\in {\Types}$, and for $B\in\mathcal{B}(\TInt)$, we set
\[
\bar{\nu}(B)=\nu(\PP\times B).
\]

The form of the action functional is fairly easy to understand, at least
in the homogeneous case.  Suppose that there is a fixed $ \pp^*\in  \Types$
and $ \pp^\NN =  \pp^*$
for all $N\in \N$ and $n\in \{1,2,\dots, N\}$ (and thus $U=\delta_\pp$).  If $\nu^*=\delta_{\pp^*}\times \omega$ for some $\omega\in \Pspace(\TInt)$, then $\frac{\partial \nu^*}{\partial U}(\pp^*)=\omega$ and one expects that
\begin{equation*}
\BP_N(d\bnu^{N}\in d\nu^{*})
\overset{N\to \infty}{\asymp}   e^{- NH\left(\omega,\mu^{ \pp^*}_{\nu^{*},0}\right)}
=e^{-N\bar H(\nu^{*})}.\end{equation*}

Consider next a heterogeneous pool with two fixed types $ \pp^*_A$
and $ \pp^*_B$.  Assume that, in a pool of size $N$, every third name is of type $ \pp^*_A$ the remaining names are of type $ \pp^*_B$.   If
\begin{equation*} \nu^* = \frac13 \delta_{\pp^*_A}\times \omega_A + \frac23 \delta_{\pp^*_B}\times \omega_B, \end{equation*}
then
\begin{equation*} \frac{\partial \nu^*}{\partial U}(\pp) =\begin{cases} \omega_A &\text{if $\pp=\pp^*_A$} \\
\omega_B &\text{if $\pp=\pp^*_B$} \end{cases}\end{equation*}
Then
\begin{align*}
\frac{\partial \nu^N}{\partial U_N}(\pp^*_A)&= \frac1{N/3}\sum_{\substack{1\le n\le N \\n\in 3\N}}\delta_{\tau_n}\chi_{\{\tau_n\le T\}} + \frac1{N/3}\sum_{\substack{1\le n\le N \\n\in 3\N}}\delta_{\star}\chi_{\{\tau_n>T\}}\\
\frac{\partial \nu^{N}}{\partial U_N}(\pp^*_B)&= \frac1{2N/3}\sum_{\substack{1\le n\le N \\n\not\in 3\N}}\delta_{\tau_n}\chi_{\{\tau_n\le T\}} + \frac1{2N/3}\sum_{\substack{1\le n\le N \\n\not\in 3\N}}\delta_{\star}\chi_{\{\tau_n>T\}}
\end{align*}

At a heuristic level, one expects that
\begin{align*}
\BP_N(d\bnu^N\in d\nu^*)&=\BP_N\lb  \frac{\partial \bnu^N}{\partial U_N}(\pp^*_A) \approx d\omega_A,\, \frac{\partial \bnu^N}{\partial U_N}(\pp^*_B) \approx d\omega_B\rb \\
&\overset{N\to \infty}{\asymp}  \exp\left[-\frac{N}{3}H\left( \omega_A,\mu_{\bar{\nu},0}^{{\pp^*_A}}\right)
-\frac{2N}{3}H\left( \omega_B,\mu_{\bar{\nu},0}^{{\pp^*_B}}\right) \right]
=e^{-N\bar{H}(\nu^{*})}.
\end{align*}

Theorem \ref{T:LDPmeasuresHeter} gives the rigorous proof that,
in general, $\bar{H}(\nu)$ is the rate function for $\{\bnu^{N},N<\infty\}$ in the heterogeneous case. An immediate consequence of Theorem \ref{T:LDPmeasuresHeter} and contraction principle is the large deviations principle for $L^{N}_{t}=\bnu^{N}(\Types\times [0,t])$.
\begin{theorem}\label{T:LDPHeterogeneous1}
Consider the system defined in \eqref{E:main} with $\eps_N=0$ and let $T<\infty$. Under Assumptions \ref{A:Bounded} and \ref{A:regularity}, the family $\{L^{N}_{T},N\in\mathbb{N}\}$ satisfies the large deviation bounds of
Definition \ref{Def:LDP}, with rate function
\[
I(\ell)=\inf\lb \bar{H}(\nu): \text{$\nu\in \Pspace(\PT)$ and $\nu\left({\Types}\times [0,T]\right)=\ell$}\rb.
\]
and speed $N$. The rate function $I$ is lower semicontinuous and has compact level sets.
\end{theorem}

As expected the rate function is essentially defined via an entropy functional, which in the present setting takes the form $H\left(\frac{\partial \nu}{\partial U}({\pp}),\mu_{\bar{\nu},0}^{{\pp}}\right)$. However, due to the heterogeneity of the environment and due to the feedback term $dL^{N}_{t}$ in the system, new phenomena appear. The effect of heterogeneity is to essentially integrate over all the different types in the $\PP$ space, whereas the  feedback term is responsible for the subscript $\bar{\nu}$ in the $\mu_{\bar{\nu},0}^{{\pp}}$, which is a non-linear effect in the entropy.

For $\xi,\varphi\in AC([0,T];\mathbb{R})$ let us define the functional
\[
g\left(\xi,f^{{\pp}}_{\varphi,0}\right)=\int_{0}^T \ln\left(\frac{\dot \xi(t)}{f^{{\pp}}_{\varphi,0}(t)}\right)\dot \xi(t)dt+ \ln \left(\frac{1-\xi(T)}{1-\int_{0}^T f^{{\pp}}_{\varphi,0}(t)dt}\right)\left(1- \xi(T)\right)
\]

Next, in Corollary \ref{C:LDPHeterogeneous2}, we note that the rate function $I$ has a straightforward  alternate representation which is a bit more suited to numerical investigations.
\begin{corollary}\label{C:LDPHeterogeneous2}
Consider the system defined in \eqref{E:main} with $\eps_N=0$. Define
\[
I'(\ell)=\inf_{\varphi \in AC({\Types}\times [0,T];\R), \varphi\geq 0}\left\{\int_{{\Types}}g\left(\varphi({\pp}),f^{{\pp}}_{\bar{\varphi},0}(t)\right)U(d{\pp}): \varphi({\pp},0)=0,\bar{\varphi}(s)=\int_{{\Types}}\varphi({\pp},s)U(d{\pp}), \textrm{ and }\bar{\varphi}(T)=\ell \right\}
\]
Under the assumptions and notation of Theorem \ref{T:LDPHeterogeneous1}, $I=I'$.
\end{corollary}

An immediate consequence of Theorem \ref{T:LDPHeterogeneous1} is the law of large numbers result for $L^{N}_{t}$ obtained in \cite{GSS2013}.
\begin{corollary}\label{C:LDPgivingLLNHeter}   Assume that $\eps_N=0$.
There is a unique measure $\nu^{*}$ such that for every $t\in[0,T]$
\[
\bar{\nu}^{*}([0,t])=1-\int_{{\Types}}e^{-\left(b^{{\pp}}(t)\lambda_{0}+\int_{0}^{t}\alpha\bar{\lambda}b^{{\pp}}(t-u)du+\beta^{C}\int_{0}^{r}b^{{\pp}}(t-u)\bar{\nu}^{*}(du)\right)}U(d{\pp});
\]
this is the unique solution of $\bar{H}(\nu)=0$.  Finally, $\bar{\nu}^{*}([0,t])=L_{t}$
for all $t\in [0,T]$, where $L$ was given in Lemma \ref{L:wconv}.
\end{corollary}
\noindent The proof of Corollary \ref{C:LDPgivingLLNHeter} is in Section \ref{SS:LDPDependentHeterogeneous}.

As an example of possible use of Theorem \ref{T:LDPHeterogeneous1} and Corollary \ref{C:LDPHeterogeneous2}, let's
consider  homogenous pools and heterogeneous pools composed of $K$ different bins that are homogeneous within each bin.

\begin{example}[Homogeneous] Fix $ \pp^*\in  \Types$ and assume
that $ \pp^\NN= \pp^*$ for all $N$ and $n$.
Then
\begin{align*}
I(\ell)&=\inf\left\{g\left(\varphi, f^{{\pp}}_{\varphi,0}\right):  \varphi(0)=0, \varphi(T)=\ell, \varphi\geq 0,\varphi\in AC([0,T];\R) \right\}.\nonumber
\end{align*}
\end{example}

\begin{example}[Heterogeneous]
Let us assume that the pool is composed of $K$ different bins.  Assume that $\kappa_{i}\%$ of the names are of type $A_{i}$ with $i=1,\cdots, K$ and $\sum_{i=1}^{K}\kappa_{i}=100$.
Setting $\varphi({\pp},s)=\sum_{i=1}^{K}\frac{\kappa_{i}}{100}\varphi_{A_{i}}(s)\chi_{\{{\pp}_{A_{i}}\}}$, we get that
\begin{align*}
I(\ell)&=\inf\left\{\sum_{i=1}^{K}\frac{\kappa_{i}}{100}g\left(\varphi_{A_{i}}, f^{{\pp}_{A_{i}}}_{\varphi,0}\right): \varphi(t)=\sum_{i=1}^{K}\frac{\kappa_{i}}{100}\varphi_{A_{i}}(t) \textrm{ for every }t\in[0,T]\right.\nonumber\\
&\qquad\qquad\qquad\left. \varphi(T)=\ell, \varphi_{A_{i}}(0)=0, \varphi_{A_{i}}\geq 0, \varphi_{A_{i}}\in AC([0,T];\R) \textrm{ for every }i=1,\cdots, K\right\}.\nonumber
\end{align*}
\end{example}

\subsection{Main result: Large deviations principle for the case $\lim_{N\rightarrow \infty}\eps_N= 0$.}\label{SS:HeterogeneousSystematic}

In this subsection, we study the case where systematic effects are present. When $\eps_N\neq 0$, the large deviations of process $\{X^{N}_{t}=\eps_N X_t, N\in\N\}$ affect the large deviations of the empirical default rate process $\{L^{N}_{t}, N\in\N\}$. In order to properly formulate our result we need to make some assumptions on the scaling properties of the coefficients of the $X$ process. These are minimal assumptions that guarantee the existence of a large deviations principle for the family $\{X^{N}_{t}=\eps_{N}X_{t}, N\in \N\}$. In particular:
\begin{assumption}\label{A:CoefficientsX}
We assume that the following limits exist uniformly on bounded subsets of $\R$
\begin{itemize}
\item{$\bar{b}(x)=\lim_{\eps\downarrow 0} b_{\eps}(x)=\lim_{\eps\downarrow 0}\eps b(x/\eps)$.}
\item{For some $\zeta\in(0,1]$, $\bar{\kappa}(x)=\lim_{\eps\downarrow 0} \kappa_{\eps}(x)=\lim_{\eps\downarrow 0}\eps^{1-\zeta} \kappa(x/\eps)$.}
\end{itemize}
Moreover, the coefficients $\bar{b}(x),\bar{\kappa}(x)$ are uniformly continuous on compact subsets of $\R$ and we assume that the SDE with drift coefficient $\bar{b}(x)$
 and diffusion coefficient $\bar{\kappa}(x)$ has a unique strong solution.
\end{assumption}
For any $u\in L^{2}\left([0,T];\R\right)$ define the map $\Gamma:L^{2}\left([0,T];\R\right)\mapsto C([0,T];\R)$  by the equation
\begin{equation}
\psi(t)=\int_{0}^{t}\bar{b}(\psi(s))ds+\int_{0}^{t}\bar{\kappa}(\psi(s))u(s)ds\label{Eq:ControlProblem}
\end{equation}
\begin{assumption}\label{A:CoefficientsX2}
We assume that for any $u\in L^{2}\left([0,T];\R\right)$ the map $\Gamma:L^{2}\left([0,T];\R\right)\mapsto C([0,T];\R)$ defined by (\ref{Eq:ControlProblem}) is well defined and (\ref{Eq:ControlProblem}) has a unique solution. Moreover, we assume that for every $N\in\N$, the map $\Gamma$ is continuous when it is restricted to the set $\{u\in L^{2}\left([0,T];\R\right): \int_{0}^{T}|u(s)|^{2}ds\leq N\}$ endowed with the weak topology of $L^{2}[0,T]$.
\end{assumption}
We need one more assumption.
\begin{assumption}\label{A:CoefficientsX3}
Let $u\in \mathcal{A}$ the set of square integrable on $[0,T]$, $\R$-valued and $\mathfrak{F}_{t}$ predictable processes and consider the controlled sde
\[
dX^{\eps,u}_t = \left[b_{\eps}( X^{\eps,u}_t)+ \kappa_{\eps}( X^{\eps,u}_t)u(t)\right]dt + \eps^{\zeta}\kappa_{\eps}(X^{\eps,u}_t)dV_t, \quad X^{\eps,u}_{0}=0.
\]
If $\lim_{n\rightarrow\infty}\eps_{n}=0$ and $\{u_{n}\}_{n\in\N}\mathcal{A}$ such that $\sup_{n\in\N}\int_{0}^{T}|u_{n}(s)|^{2}ds\leq N$ almost surely, then $X^{\eps_{n},u_{n}}$ is tight in $C\left([0,T];\R\right)$ and
\[
\sup_{n\in\N}\mathbb{E}\int_{0}^{T}|\bar{\kappa}(X^{\eps_{n},u_{n}}_{s})|^{2}ds<\infty.
\]
\end{assumption}

As we shall see in Lemma \ref{L:LDP_X} of Subsection \ref{SS:LDPDependentHeterogeneous2}, under Assumptions \ref{A:CoefficientsX}, \ref{A:CoefficientsX2} and \ref{A:CoefficientsX3}, the large deviations principle for the family $\{X^{N}_{t}=\eps_{N}X_{t}, N\in \N\}$ on $C\left([0,T];\R\right)$ with speed $1/\eps^{2\zeta}_{N}$ is the same as the large deviations principle and with the same speed for the family $\{\bar{X}^{N}_{t}, N\in \N\}$, where
\begin{equation}
d\bar{X}^{N}_t = \bar{b}( \bar{X}^{N}_t) dt + \eps_{N}^{\zeta}\bar{\kappa}(\bar{X}^{N}_t)dV_t, \quad \bar{X}^{N}_{0}=0.\label{Eq:EquivalentAlternativeXprocess}
\end{equation}

In this case,  the large deviations action functional for $\{\bar{X}^{N}_{t}, N\in \N\}$ in $C([0,T];\R)$ is
\begin{equation}
J_{X}(\psi)=\inf\left\{\frac{1}{2}\int_{0}^{T}\left|u(s)\right|^{2}ds: u\in L^{2}\left([0,T];\R\right), \Gamma(u)=\psi \right\}\label{Eq:RateFunctionXprocesses}
\end{equation}
whenever  $\{u\in L^{2}\left([0,T];\R\right), \Gamma(u)=\psi\}\neq\emptyset$ and $J_{X}(\psi)=\infty$ otherwise.

We remark here that if $\bar{\kappa}(x)\neq 0$ for all $x\in\R$, then for $\psi\in AC\left([0,T];\R\right)$ with $\psi(0)=0$, we have the simplified
 well known form, see Section 5.3 of \cite{FWBook},
\begin{equation*}
J_{X}(\psi)=\frac{1}{2}\int_{0}^{T}\left|\frac{\dot{\psi}(t)-\bar{b}(\psi(s))}{\bar{\kappa}(\psi(s))}\right|^{2}ds
\end{equation*}
and $J_{X}(\psi)=\infty$ otherwise. However, in general, the form of the rate function is given by \eqref{Eq:RateFunctionXprocesses}. 

\begin{example}\label{E:LDP_Xprocess}
Two classical examples, where  Assumptions \ref{A:CoefficientsX}, \ref{A:CoefficientsX2} and \ref{A:CoefficientsX3} hold and thus the LDP  for the process $\{{X}^{N}, N\in\N\}$ holds, are (a): the Ornstein-Uhlenbeck process with
$b(x)=-\gamma x$ and $\kappa(x)=1$, where then $\bar{b}(x)=-\gamma x$, $\bar{\kappa}(x)=1$ and $\zeta=1$, and (b): the CIR (or square-root) process with
$b(x)=-\gamma (x-\bar{x}/\eps)$ and $\kappa(x)=\sqrt{x}$, where then $\bar{b}(x)=-\gamma (x-\bar{x})$, $\bar{\kappa}(x)=\sqrt{x}$ and $\zeta=1/2$.
\end{example}

Then, we have the following theorem.

\begin{theorem}\label{T:LDPHeterogeneous2}
Consider the system defined in \eqref{E:main} with $\lim_{N\rightarrow\infty}\eps_{N}=0$ such that $\lim_{N\rightarrow\infty}N\eps^{2\zeta}_{N}= c\in(0,\infty)$
and let $T<\infty$. Under Assumptions \ref{A:Bounded}, \ref{A:regularity}, \ref{A:CoefficientsX}, \ref{A:CoefficientsX2} and \ref{A:CoefficientsX3} the family $\{L^{N}_{T},N\in\mathbb{N}\}$ satisfies the large deviation bounds of
Definition \ref{Def:LDP}, with speed $N$ and with rate function
\[
I(\ell)=\inf\left\{ S(\varphi,\psi): \varphi\in C\left({\Types}\times [0,T];R\right),\psi\in C\left([0,T];R\right), \bar{\varphi}(T)=\ell \right\}
\]

where
\begin{equation*} \label{E:exLDP}
\begin{aligned}
S(\varphi,\psi)&=
\begin{cases}
\int_{{\Types}}g\left(\varphi({\pp}),f^{{\pp}}_{\bar{\varphi},\psi}(t)\right)U(d{\pp})+\frac{1}{c}J_{X}(\psi), &\text{if $\varphi\in AC\left({\Types}\times [0,T];\R\right),\psi\in AC\left([0,T];\R\right), \psi(0)=0,$}\\
 &\text{$\varphi({\pp},0)=0, \varphi\geq 0,\bar{\varphi}(s)=\int_{{\Types}}\varphi({\pp},s)U(d{\pp})$}\\
\infty, &\text{otherwise}
\end{cases}
\end{aligned}\end{equation*}

Here, $J_{X}(\psi)$ is the rate function for the process $\{{X}^{N}, N<\infty\}$, as defined by \eqref{Eq:RateFunctionXprocesses}.
$I(\ell)$ has compact level sets.
\end{theorem}

The definition of $S(\varphi,\psi)$ suggests that if $c=\lim_{N\rightarrow\infty}N\eps^{2\zeta}_{N}=\infty$ then the $\int_{{\Types}}g\left(\varphi({\pp}),f^{{\pp}}_{\bar{\varphi},\psi}(t)\right)U(d{\pp})$
integral in $S(\varphi,\psi)$ will be the dominant factor, whereas if $c=\lim_{N\rightarrow\infty}N\eps^{2\zeta}_{N}=0$, then the $J_{X}(\psi)$
entropy term will be the dominant factor. This will become clearer in Section \ref{SS:LDPDependentHeterogeneous2}. Hence both effects are preesent if
$c=\lim_{N\rightarrow\infty}N\eps^{2\zeta}_{N}\in (0,\infty)$ and this is the case that we focus on in this paper. However, in the course of the proof of Theorem
\ref{T:LDPHeterogeneous2} we prove Lemma \ref{L:LDP1}, which is  the large deviations principle for
$\{L^{N}_{t}, N\in\N, t\in[0,T]\}$ conditional a given path of the process $t\mapsto X_{t}$ in $C_{c}([0,T;\R])$, when $\eps_{N}=1$.

\section{Numerical exploration of the large deviations principle}\label{S:Numerics}
In this section we illustrate our theoretical results by some numerical computations. In particular, in Subsection \ref{SS:NumericsHomogeneous} we investigate how the rate function, the tails of the probability loss distribution and the extremals behave in specific situations for a homogeneous pool. We compare two different portfolios and qualitatively compare the two cases. Then, in Subsection \ref{SS:NumericsHeterogeneous} we perform some numerical experiments for a heterogeneous pool composed of two types, where the difference in the two types is in the level of influence of contagion and systematic effects.

Understanding  the most likely ways in which contagion and systematic risk combine to lead to large default rates, gives useful insights into how to optimally hedge against such events. In particular, the numerical experiments show that if a large cluster were to occur, the effect of the systematic factor would be more significant in the initial phase, but then its importance decreases and contagion effects become more important. Moreover, in the case of heterogeneous pools, the large deviations analysis allows us to make statements about which types of names in the pool are likely to be affected more and in which order.

In the numerical results that follow we have computed numerically the rate function $I(\ell)$ and the corresponding extremals $\varphi$ and $\psi$ solving the variational problems of Theorem \ref{T:LDPHeterogeneous2} for the homogeneous and heterogeneous cases of interest. In the numerical computation of the rate function, one needs to compute $f_{\varphi,\psi}^{\pp}(t)$ and this is done with the help of Lemma \ref{L:DensityFcn}.
\subsection{Numerics for homogeneous pool}\label{SS:NumericsHomogeneous}
Let us now understand how our calculations look like in some specific cases in the case of homogeneous portfolios.  We consider the systematic risk
$X_{t}$ to be of Ornstein-Uhlenbeck type and in particular
\begin{equation}\begin{aligned}
dX_t&=-\gamma X_tdt+dV_t\\
X_0&=0\end{aligned}\label{Eq:SpecificExogeneousRisk}
\end{equation}

For the numerical experiments below we have taken $\eps_{N}=\frac{1}{\sqrt{N}}$. The LDP is given by Theorem \ref{T:LDPHeterogeneous2} with $c=1$.
The main difficulty in evaluating $I(\ell)$ and the extremals is the computation of $\mu_{\varphi,\psi}^{\pp}$ of \eqref{E:muDef}.  The virtue of the  CIR-based
evolution of $\lambda^\NN$ of \eqref{E:main} is that fairly explicit formulae are available.  For $t\in[0,T]$, let $\theta_t^{\pp}$ solve
\begin{equation} \label{E:DuffiePanSingleton2}
 \theta_t^{\pp}(s) = \int_{0}^s\left(1-\frac{1}{2}\sigma^2 \left(\theta_t^{\pp}(r)\right)^2-\alpha\theta_t^{\pp}(r)\right)dr +\beta^S\int_{0}^s \theta_t^{\pp}(t-r)d\psi(r). \qquad s\in [0,t]
\end{equation}
Then define
\begin{equation*}
\Gamma_{\varphi,\psi}^{\pp}(t) = \theta_t^{\pp}(t)\lambda_\circ + \alpha\bar\lambda  \int_{0}^t \theta_t^{\pp}(r)dr+ \beta^C \int_{0}^t \theta_t^{\pp}(t-r)d \varphi(r) \end{equation*}
\begin{lemma} \label{L:DensityFcn} We have that
\begin{equation*}f_{\varphi,\psi}^{\pp}(t) =\dot \Gamma_{\varphi,\psi}^{\pp}(t)\exp\left[-\Gamma_{\varphi,\psi}^{\pp}(t)\right]. \end{equation*}
\end{lemma}
\begin{proof}
Define
\begin{equation*}
M_s \Def \exp\left[-\theta^{\pp}_t(t-s)\lambda^{\varphi,\psi}_s- \alpha\bar\lambda\int_{s}^t \theta^{\pp}_t(t-r)dr -\beta^C\int_{s}^t   \theta^{\pp}_t(t-r)d\varphi(r)-\int_{0}^s \lambda^{\varphi,\psi}_r dr\right] \end{equation*}
for $s\in [0,t]$.
Note that $\theta^{\pp}_t(0)=0$.  In differential form,
\begin{align*}
dM_s&= \lb \dot \theta^{\pp}_t(t-s)\lambda^{\varphi,\psi}_s - \theta^{\pp}_t(t-s)\lb -\alpha(\lambda^{\varphi,\psi}_s-\bar \lambda) + \beta^C \dot \varphi(s)+\beta^S \dot \psi(s)\lambda^{\varphi,\psi}_s\rb - \lambda^{\varphi,\psi}_s\right.\\
&\qquad \left. + \frac12 (\theta^{\pp}_t(t-s))^2\sigma^2 \lambda^{\varphi,\psi}_s+ \theta^{\pp}_t(t-s)\lb \alpha\bar \lambda + \beta^C \dot \varphi(s)\rb\rb  M_s ds + d\mart_s
 \end{align*}
where $\mart$ is a martingale.  The ODE \eqref{E:DuffiePanSingleton2} implies that the $ds$ term is identically zero, so $M$ is a martingale.
Noting that
\begin{equation*} M_0 = \exp\left[-\Gamma_{\phi,\psi}(t)\right] \qquad \text{and}\qquad
M_t = \exp\left[-\int_{0}^t \lambda^{\varphi,\psi}_r dr\right], \end{equation*}
we get that
\begin{equation*} e^{-\Gamma_{\psi,\psi}(t)} = M_0 = \BE[M_t] = \BE\left[\exp\left[-\int_{0}^t \lambda^{\varphi,\psi}_r dr\right]\right].\end{equation*}
Comparing this with \eqref{E:fdens} and differentiating, the claim follows.
\end{proof}

We consider two test portfolios; see Table \ref{parameters}. For each test portfolio, we compare four different cases,
(a) Independence: $\beta^S=\beta^C=0$, (b) Contagion only: $\beta^S=0, \beta^C\not=0$, (c) Systematic risk only: $\beta^S\not=0, \beta^C=0$,
and (d) Systematic risk and contagion: $\beta^S\not=0, \beta^C\not=0$. In each case, the time horizon is $T=1$.

\begin{table}[ht!]
\centering
\begin{tabular}{|l|c|c|c|c|c|c|c|c|}
\hline
 & $N$ & $\alpha$ & $\bar\lambda$ & $\sigma$ & $\lambda_{0}$ & $\gamma$ & $\beta^S$ & $\beta^C$  \\ \hline\hline
Portfolio I & 200 & 1 & 1 & 0.9 & 0.5 & 1 & 10 & 3 \\ \hline
Portfolio II & 200 & 5 & 1 & 1 & 0.5 & 0.1 & 28 & 1 \\ \hline
\end{tabular}
\vspace{0.5cm}
\caption{\label{parameters} Model parameter values for two test portfolios.}
\end{table}

 In Table  \ref{ratefcn1} we report the losses in the pool at time $T=1$ for the two portfolios, both in the case that there is a contagion effect
(i.e. $\beta^{C}\neq0$ and taking the values of Table \ref{parameters}) and in the case that there is no contagion effect (i.e., setting $\beta^{C}=0$).
\begin{table}[ht!]
\centering
\begin{tabular}{|l||c|c|}
\hline
 &  $L_{c}(1)$ & $L_{nc}(1)$ \\ \hline\hline
Portfolio I  & 0.804 & 0.470\\ \hline
Portfolio II  & 0.650 & 0.589\\ \hline
\end{tabular}
\vspace{0.5cm}
\caption{\label{ratefcn1} Typical default rate $\bar{\ell}=L_{c}(T)$ at $T=1$ when contagion is present, i.e.,
$\beta^{C}\neq 0$ and $\bar{\ell}=L_{nc}(T)$ at $T=1$ when contagion is not present, i.e., $\beta^{C}= 0$. }
\end{table}

In Figure \ref{fig:rate-functions} we see a comparison of all rate functions $I$. Note that in both cases $\bar{\ell}$ satisfies $I(\bar{\ell})=0$, where the law of
 large numbers is
numerically given by Table \ref{ratefcn1}. Of course, this is expected and is in accordance with large deviations theory.
\begin{figure}[ht!]
\begin{minipage}{8.1cm}
\includegraphics[height=8.2cm, width=6cm, angle=270]{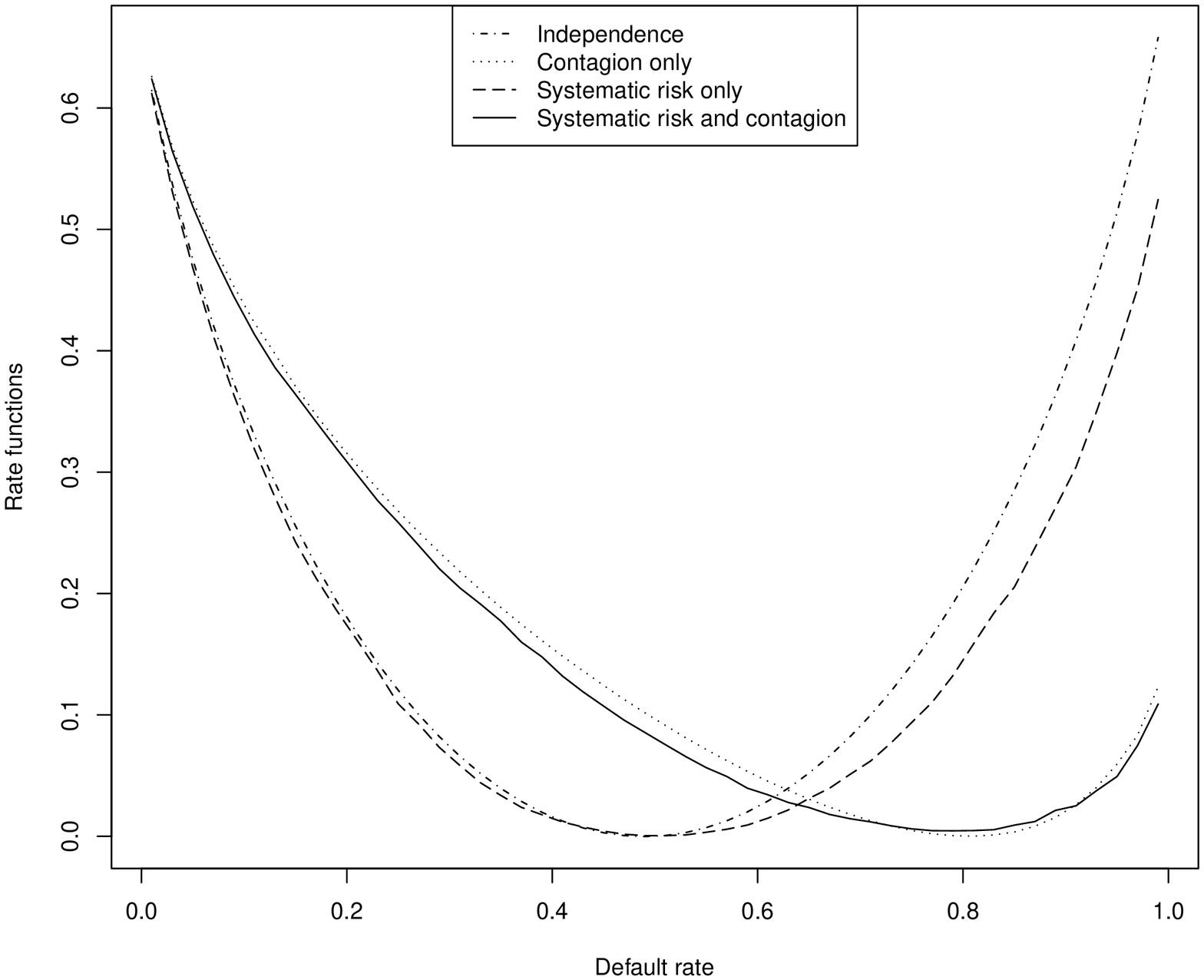}
\end{minipage}
\begin{minipage}{8.1cm}
\includegraphics[height=8.2cm, width=6cm, angle=270]{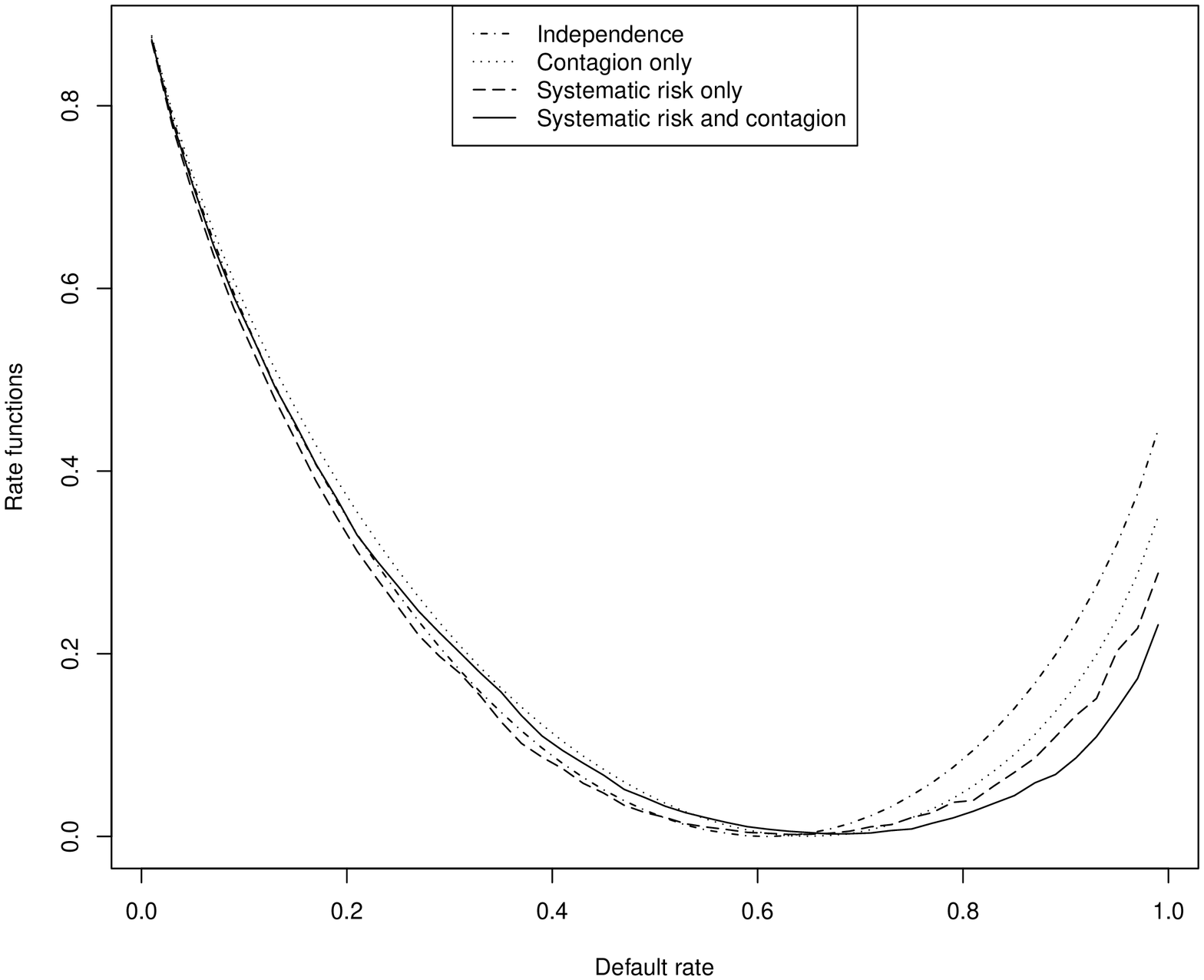}
\end{minipage}
\caption{\label{fig:rate-functions} The rate function $I'(\ell)$. \emph{Left panel}: Portfolio I. \emph{Right panel}: Portfolio II.}
\end{figure}

In Figure \ref{fig:tails} we see a large deviations approximation to the tail of $L^N_T$.
That is, we plot  $\BP\{L^N_T\approx \ell\} \asymp \exp\left[-N I(\ell)\right]$ where $\ell>\bar{\ell}=L_{T}$. The exposure to contagion and the systematic risk has significant
 implications for the tail of $L^N_T$. Moreover, we see here that, in terms of the tail of the distribution, in Portfolio I,
the effect of contagion dominates the effect of systematic risk, whereas, in Portfolio II, the effect of systematic risk dominates the effect of contagion. This is mainly due to the
differences in the values of $\beta^{S}$ and $\beta^{C}$ in the two portfolios, Table \ref{parameters}.
\begin{figure}[ht!]
\begin{minipage}{8.1cm}
\includegraphics[height=8.2cm, width=6cm, angle=270]{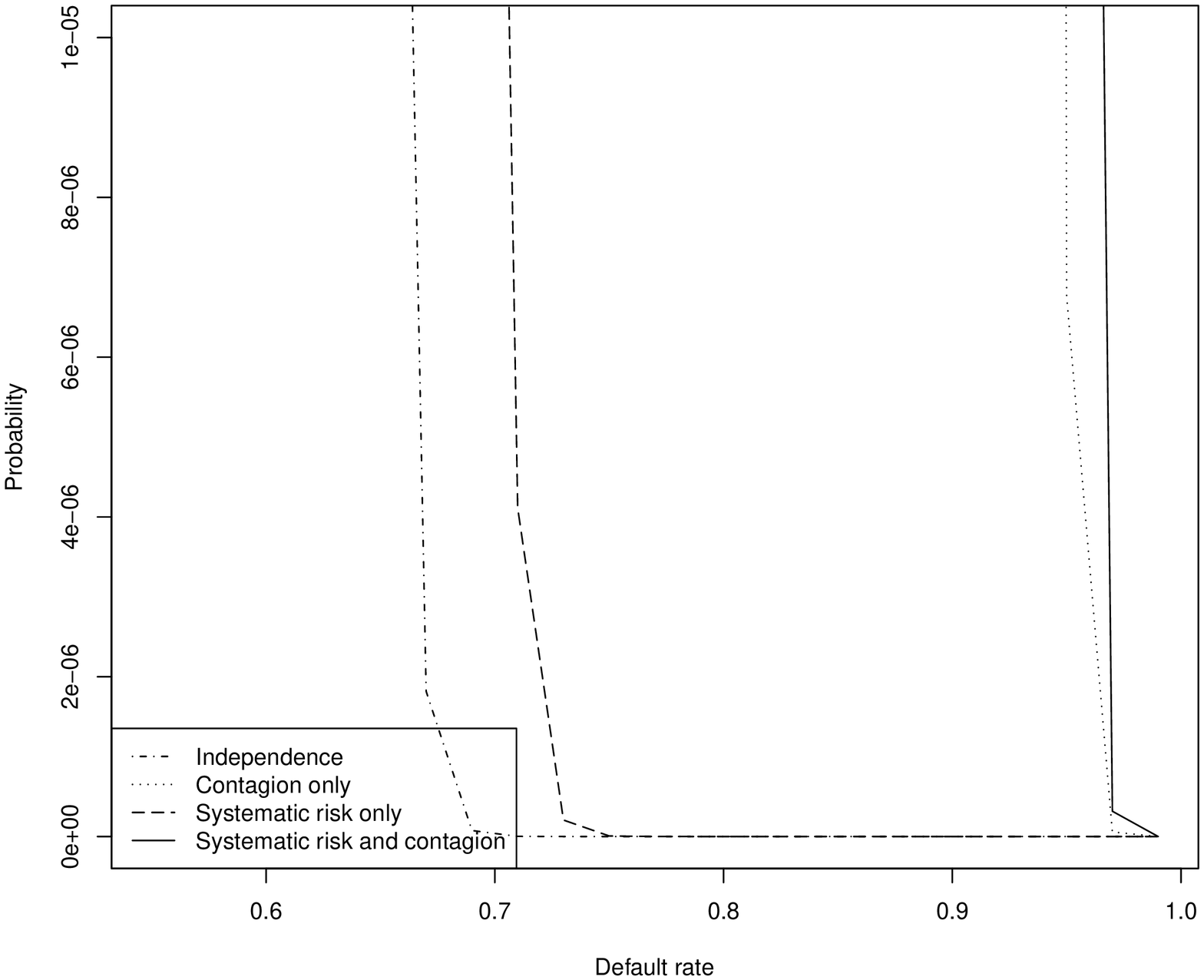}
\end{minipage}
\begin{minipage}{8.1cm}
\includegraphics[height=8.2cm, width=6cm, angle=270]{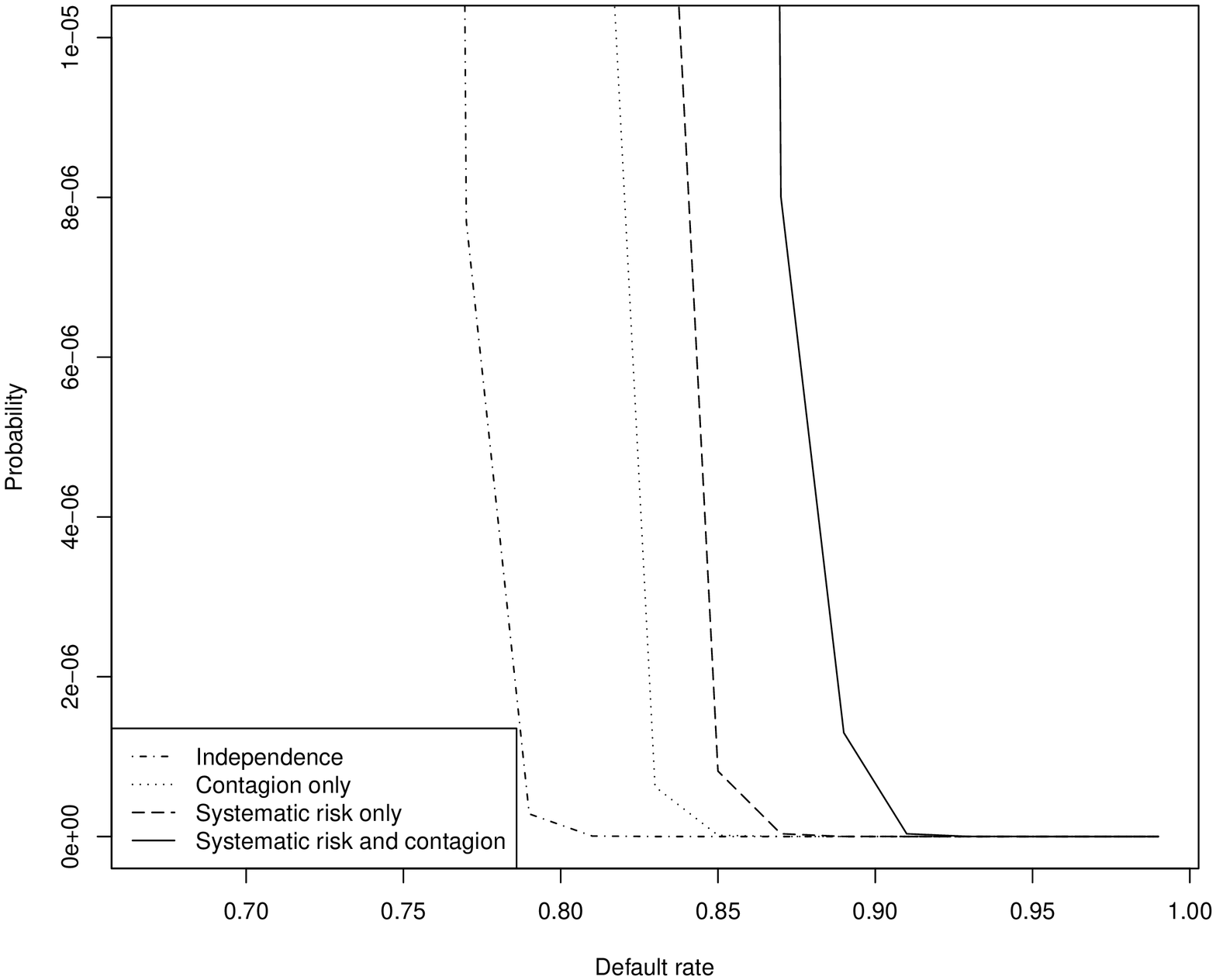}
\end{minipage}
\caption{\label{fig:tails} Large deviations approximation to the tail of the default rate $L^N_T$. \emph{Left panel}: Portfolio I. \emph{Right panel}: Portfolio II.}
\end{figure}

In Figures \ref{fig:phi} and \ref{fig:psi} we see a comparison of the optimal $\varphi$ and $\psi$ for a level $\ell=0.85$. The effect of contagion and of the systematic risk alter the behavior of the extremals (the most likely path to failure). Comparing the tails of the loss distributions of portfolios I and II in Figure  \ref{fig:tails}, we conclude that in Portfolio I, the effect of contagion is in general more profound than the effect of systematic risk. On the other hand in the case of Portfolio II, the effect of systematic risk is more profound than the effect of contagion. Moreover, in the left panel of Figure \ref{fig:psi}, we see that if a large default cluster occurs, the systematic risk is most likely to play a large role in the initial phase, but then its importance decreases (and thus the contagion effect becomes more important).
\begin{figure}[ht!]
\begin{minipage}{8.1cm}
\includegraphics[height=8.2cm, width=5cm, angle=270]{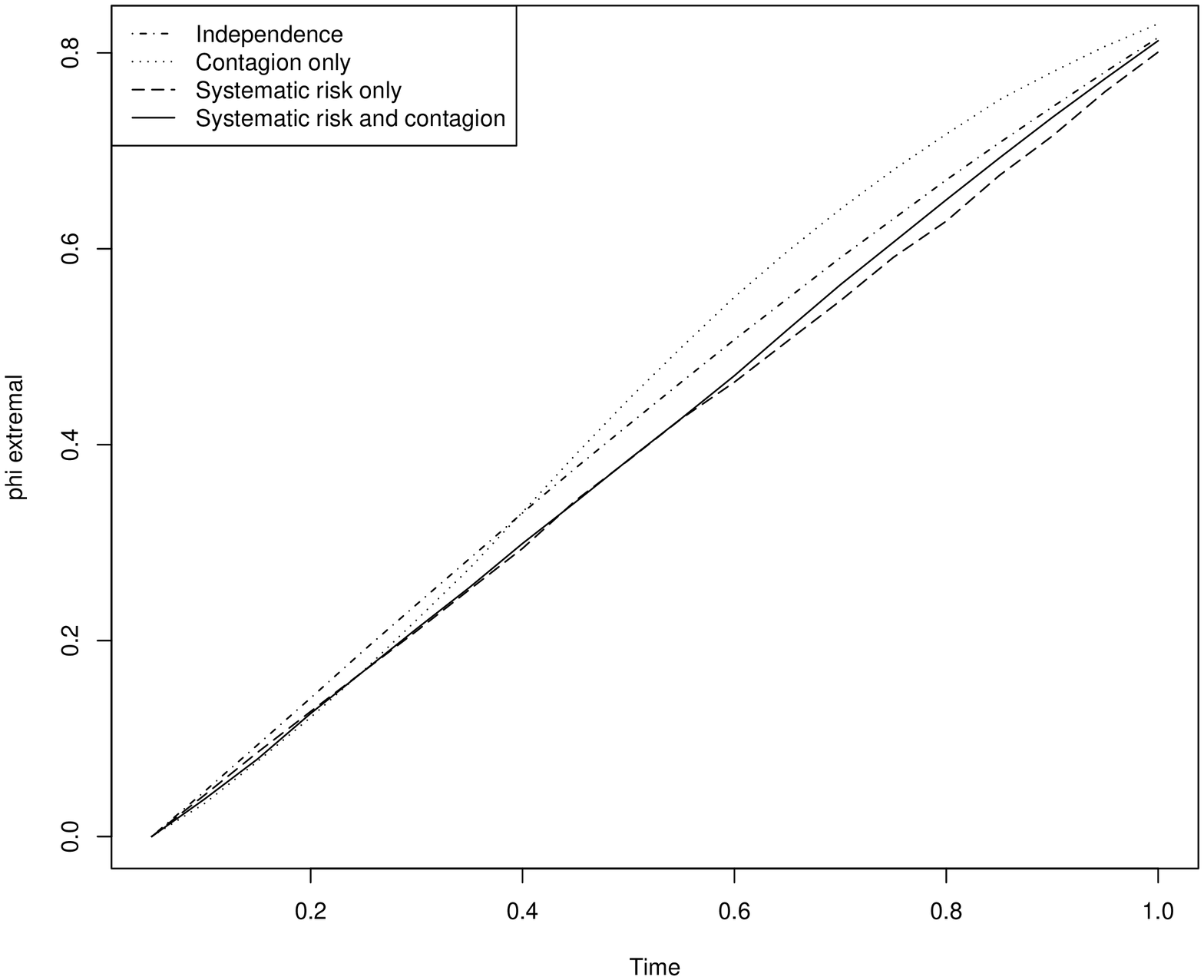}
\end{minipage}
\begin{minipage}{8.1cm}
\includegraphics[height=8.2cm, width=5cm, angle=270]{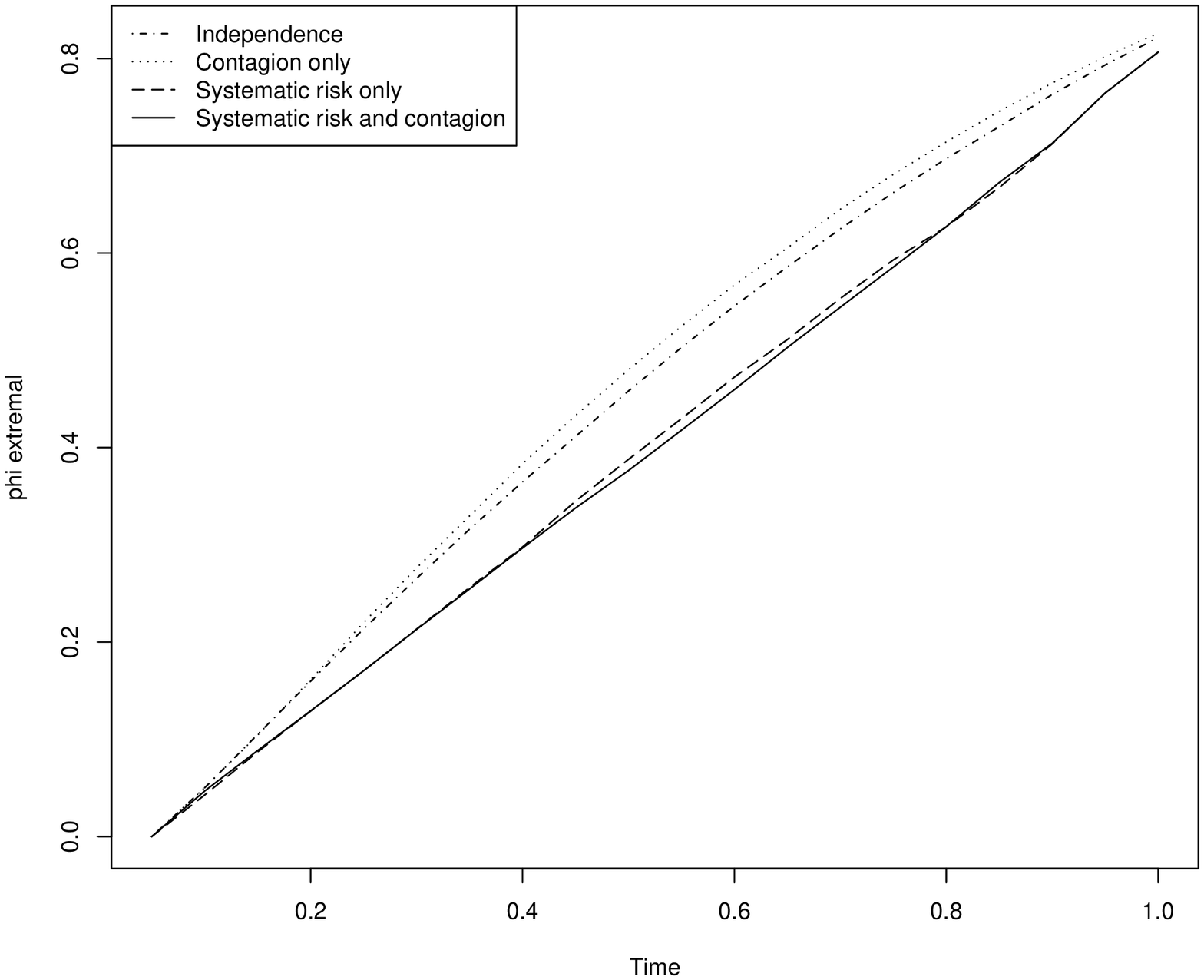}
\end{minipage}
\caption{\label{fig:phi} Optimal $\phi(t)$ for $t\in[0,1]$ and $\ell=0.85$. \emph{Left panel}: Portfolio I. \emph{Right panel}: Portfolio II.}
\end{figure}
\begin{figure}[ht!]
\begin{minipage}{8.1cm}
\includegraphics[height=8.2cm, width=5cm, angle=270]{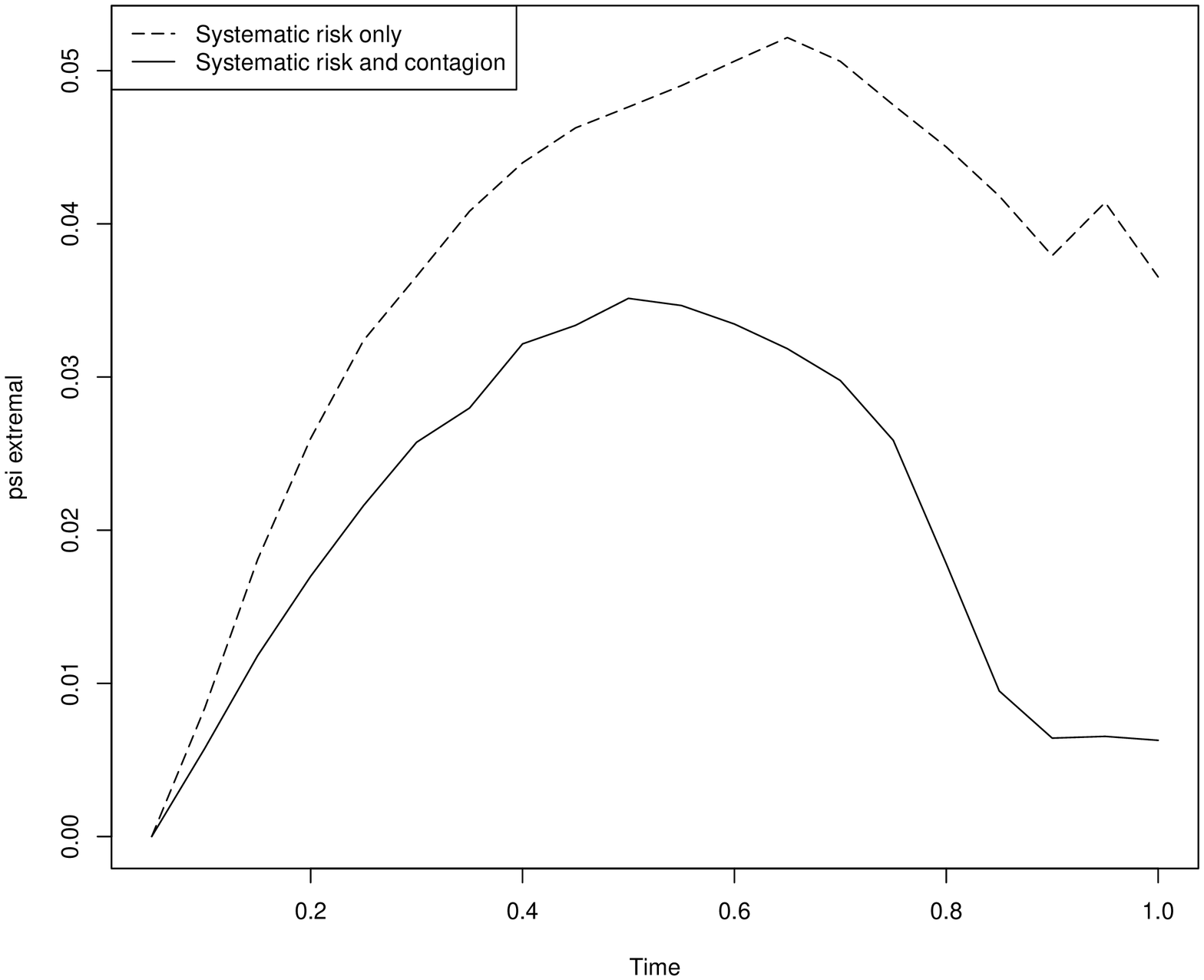}
\end{minipage}
\begin{minipage}{8.1cm}
\includegraphics[height=8.2cm, width=5cm, angle=270]{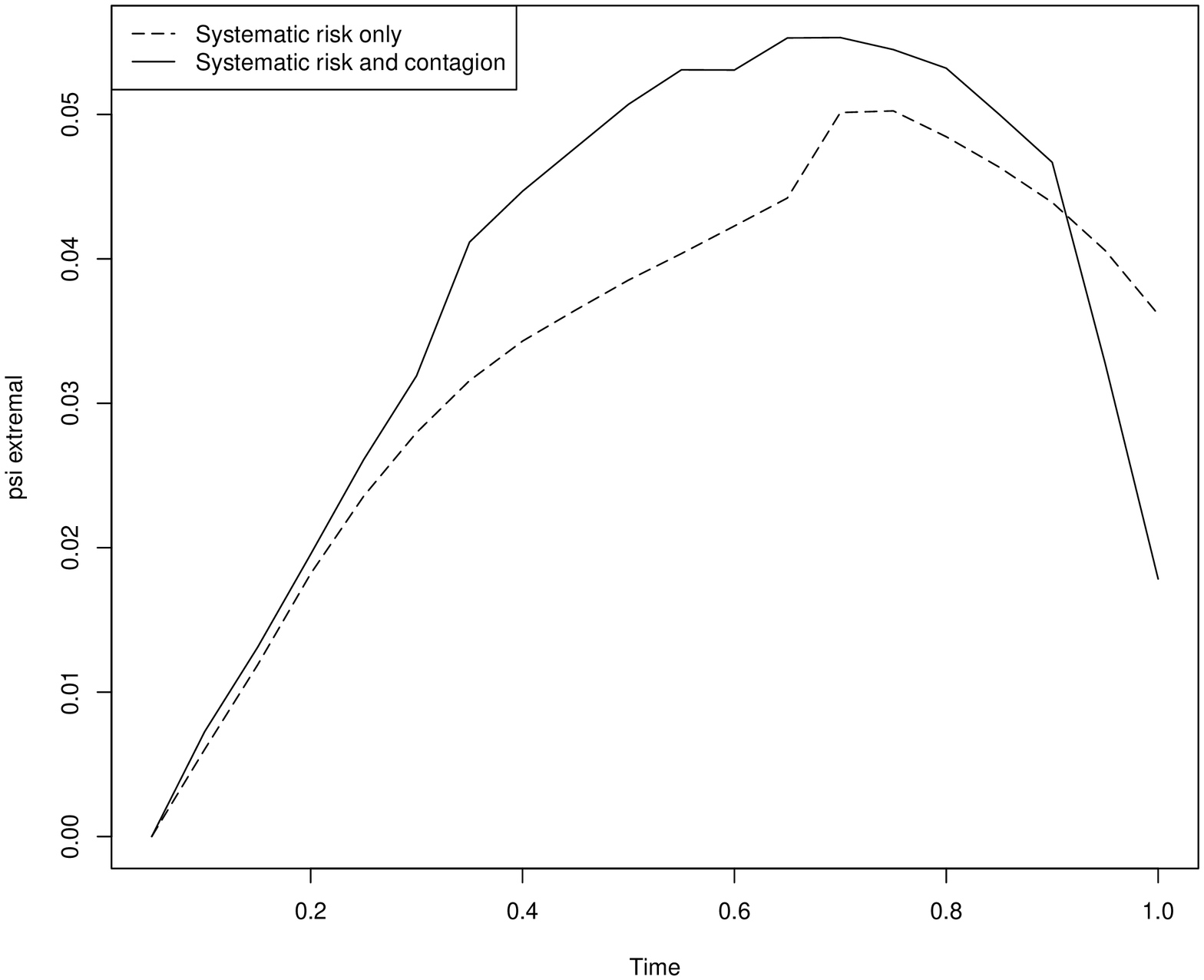}
\end{minipage}
\caption{\label{fig:psi} Optimal $\psi(t)$ for $t\in[0,1]$ and $\ell=0.85$. \emph{Left panel}: Portfolio I. \emph{Right panel}: Portfolio II.}
\end{figure}
\subsection{Numerics for heterogeneous pool}\label{SS:NumericsHeterogeneous}
In this subsection, we illustrate numerically the case of a heterogeneous portfolio, which is composed of two types.  Type A is $1/3$ of the names and type B is $2/3$ of
 the names. In order to illustrate the effect of contagion and systematic risk in such a heterogeneous portfolio, we keep all parameters the same, except for the parameters
 $(\beta^{C}_{A},\beta^{S}_{A})$ and $(\beta^{C}_{B},\beta^{S}_{B})$ for types $A$ and $B$ respectively. The systematic risk is assumed to be given by \eqref{Eq:SpecificExogeneousRisk}.

In particular, we assume that initially we have $N=200$ names in the pool and  we consider a portfolio composed of two types with the following parameters.
\begin{table}[ht!]
\centering
\begin{tabular}{|l|c|c|c|c|c|c|c|}
\hline
 &  $\alpha$ & $\bar\lambda$ & $\sigma$ & $\lambda_{0}$ & $\gamma$ & $\beta^S$ & $\beta^C$  \\ \hline\hline
Type A  &  1 & 2 & 1 & 0.5 & 1 & 5 & 10 \\ \hline
Type B  &  1 & 2 & 1 & 0.5 & 1 & 1 & 2  \\ \hline
\end{tabular}
\vspace{0.5cm}
\caption{\label{parameters2} Model parameter values for two types in the portfolios.}
\end{table}
In such a portfolio we compute that the typical loss at time $T=1$ is $0.62$ if there is no contagion and $0.81$ if there is contagion. Below, we see plots comparing the rate functions, tails of the distribution and extremals in all cases, depending on weather both contagion and systematic risk effects are present or not.

In Figure \ref{fig:rate-functionsHeter} we see a comparison of all rate functions $I$ and of the tails of the loss distribution.
\begin{figure}[ht!]
\begin{minipage}{8.1cm}
\includegraphics[height=8.2cm, width=6cm, angle=270]{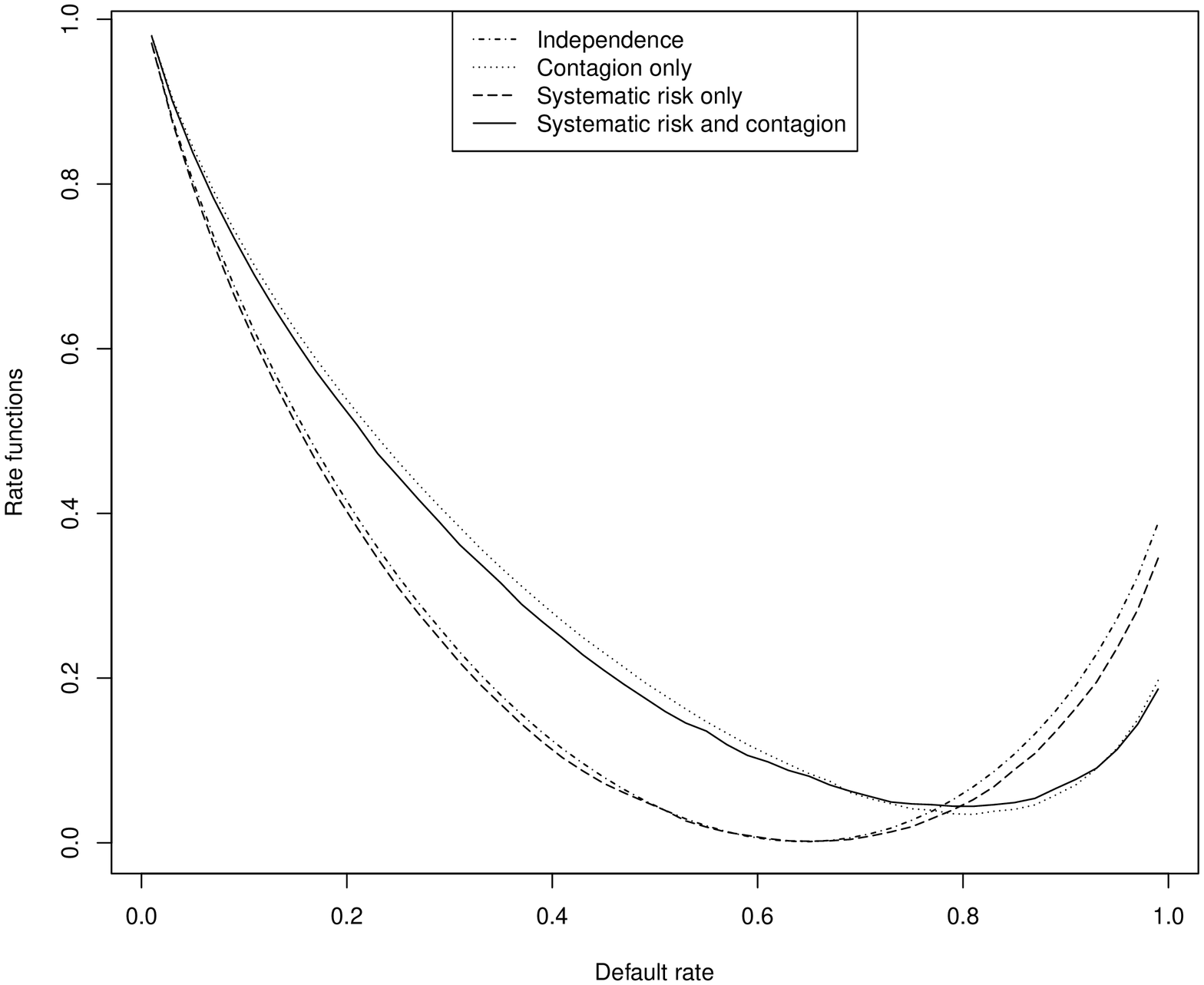}
\end{minipage}
\begin{minipage}{8.1cm}
\includegraphics[height=8.2cm, width=6cm, angle=270]{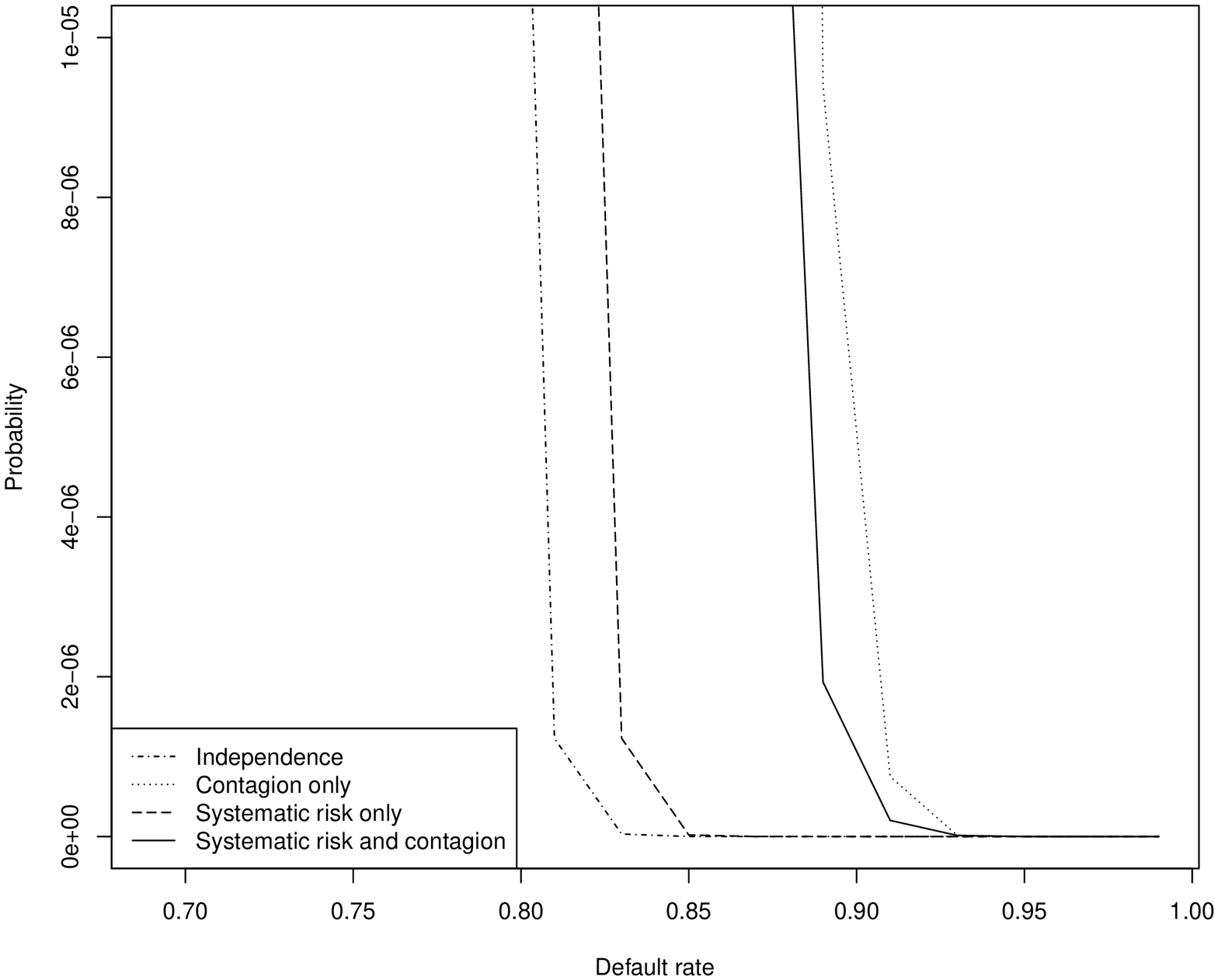}
\end{minipage}
\caption{\label{fig:rate-functionsHeter} The rate function $I(\ell)$ and the tails of the loss distribution}
\end{figure}

The effect of contagion and of the systematic risk alter the behavior of the extremals (the most likely path to failure). In Figure \ref{fig:phiHeter1} we see a comparison of the optimal $\varphi$ for a level of loss $\ell=0.85$ for types A and B under the different scenarios of contagion effects
or not and systematic risk effects or not. In order to illustrate the difference between the behavior of the most likely path to failure between types A and B, we compare the
corresponding $\varphi$ extremals in the left panel of Figure \ref{fig:phiHeter2} when both effects of contagion and systematic risk are included in the model.  We notice that at any given time $t$, the extremal for type A is bigger than the extremal for
type B. This implies that unlikely large losses for components of type A are more likely than unlikely large losses for components of type B. Thus, components of
 type A affect the pool more than components of type B even though type A composes 1/3 of the pool, whereas type B, composes 2/3 of the pool.

Then, in the right of Figure \ref{fig:phiHeter2}  we see a comparison of the optimal $\psi$ extremals for a level $\ell=0.85$ for the cases of systematic risk effects are present or not.
  We see that in the presence of contagion, if a large default cluster occurs, the systematic risk is most likely to play a large role in the initial phase, but then its importance decreases, which means that contagion effect becomes more important. So, in conclusion the analysis of the extremals shows that in such a pool if a large cluster were to occur, it would most likely be due to effects of the systematic risk factor which then affects more names of type A and less names of type B.
\begin{figure}[ht!]
\begin{minipage}{8.1cm}
\includegraphics[height=8.2cm, width=5cm, angle=270]{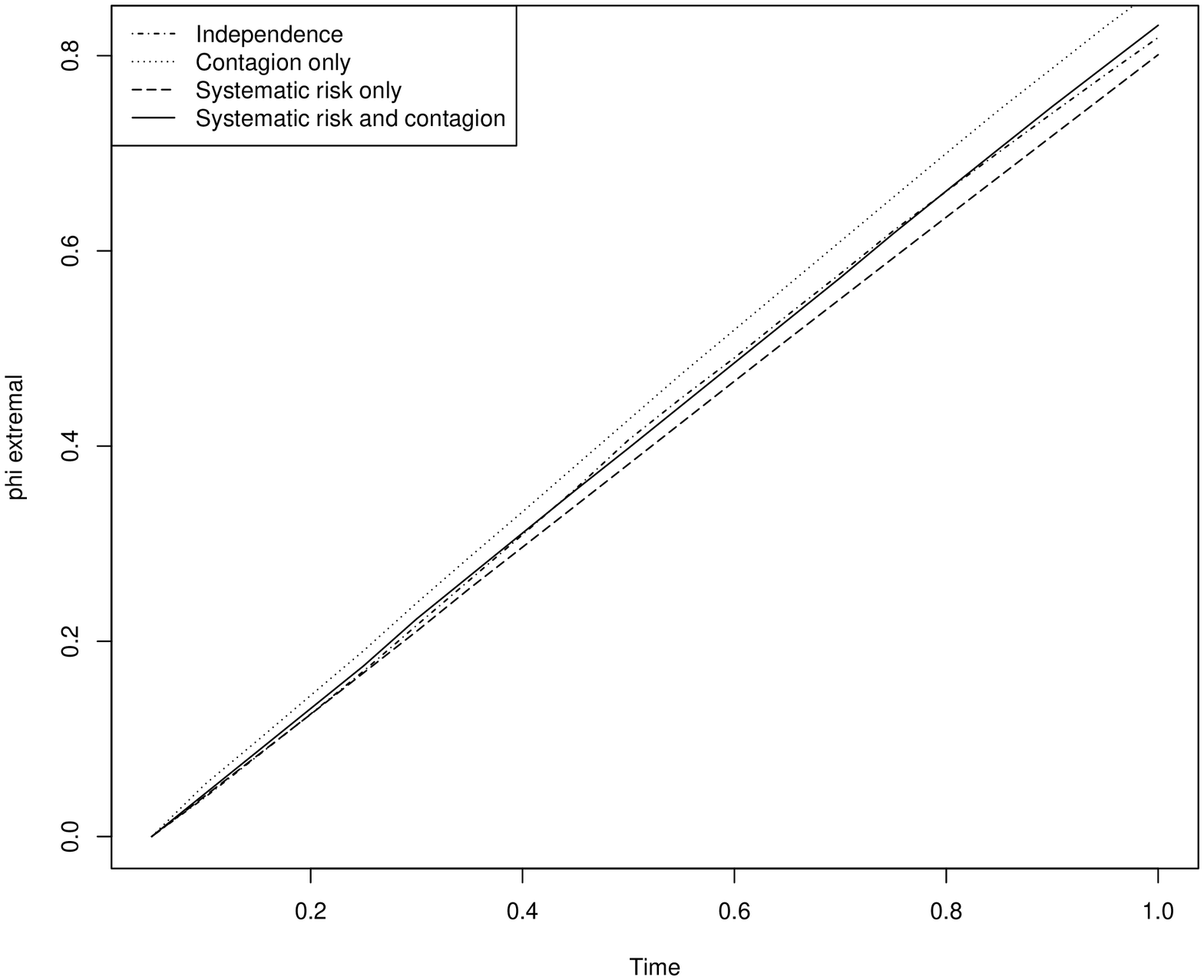}
\end{minipage}
\begin{minipage}{8.1cm}
\includegraphics[height=8.2cm, width=5cm, angle=270]{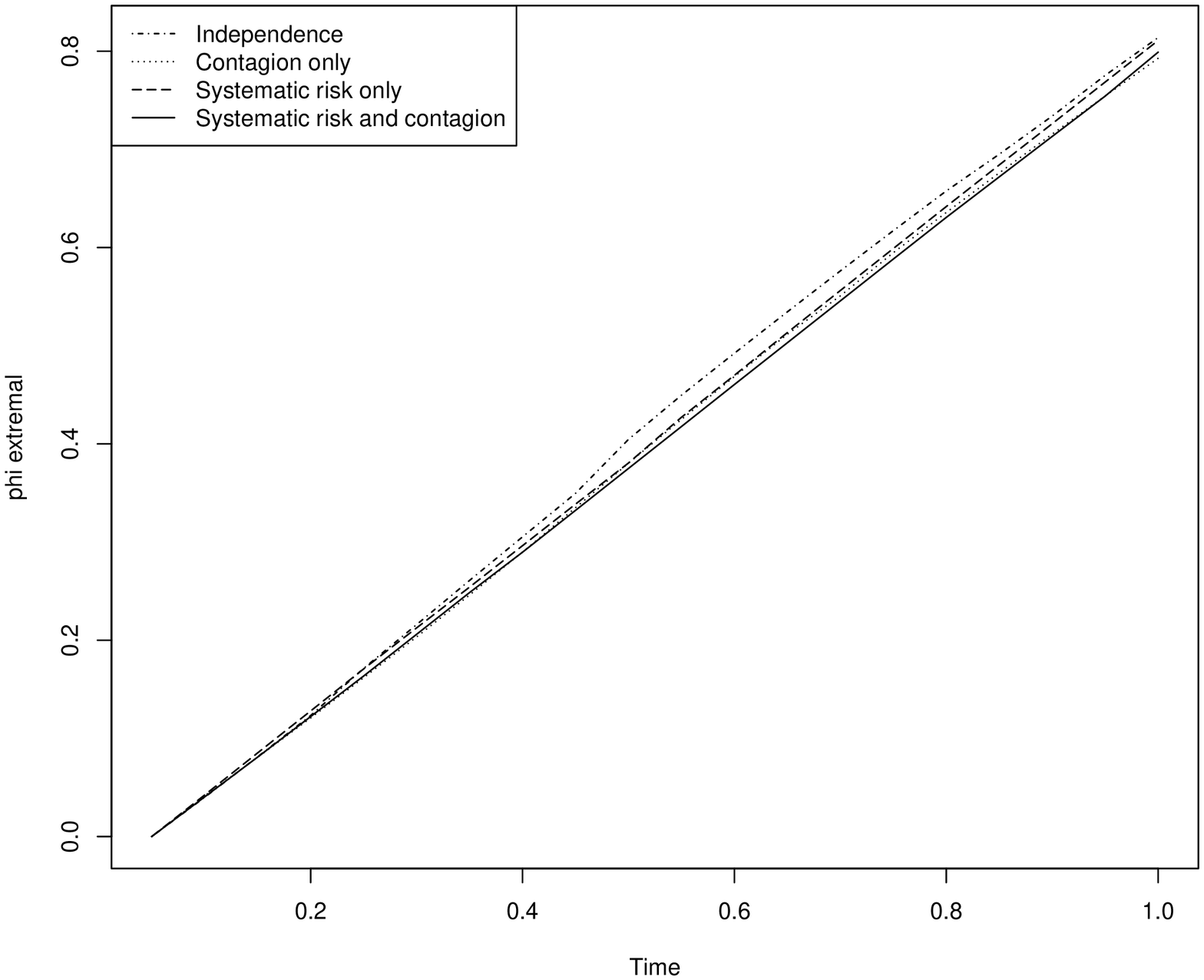}
\end{minipage}
\caption{\label{fig:phiHeter1} Optimal $\phi(t)$ for $t\in[0,1]$ and $\ell=0.85$. \emph{Left panel}: $\phi$ extremal for type A. \emph{Right panel}:  $\phi$ extremal for type B.}
\end{figure}
\begin{figure}[ht!]
\begin{minipage}{8.1cm}
\includegraphics[height=8.2cm, width=5cm, angle=270]{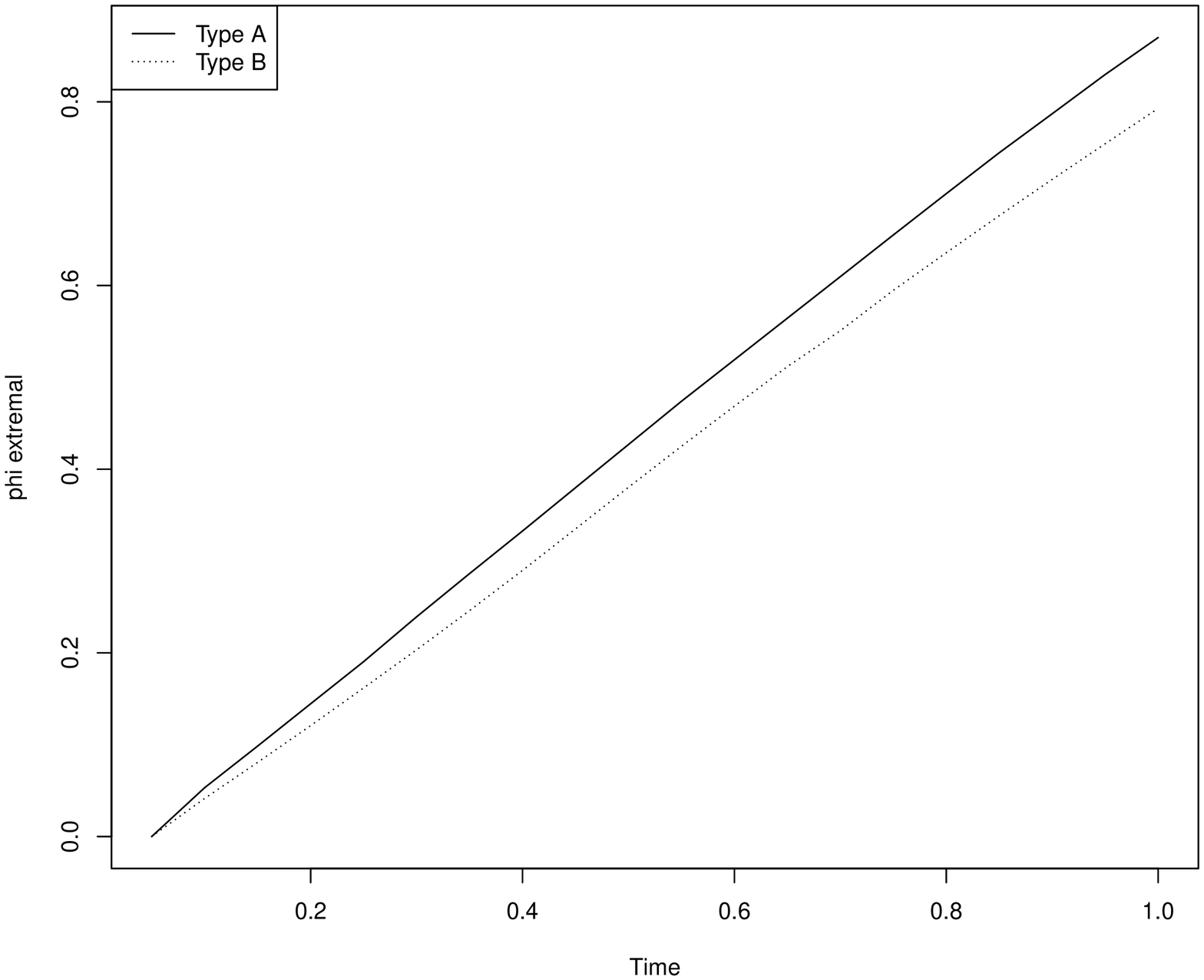}
\end{minipage}
\begin{minipage}{8.1cm}
\includegraphics[height=8.2cm, width=5cm, angle=270]{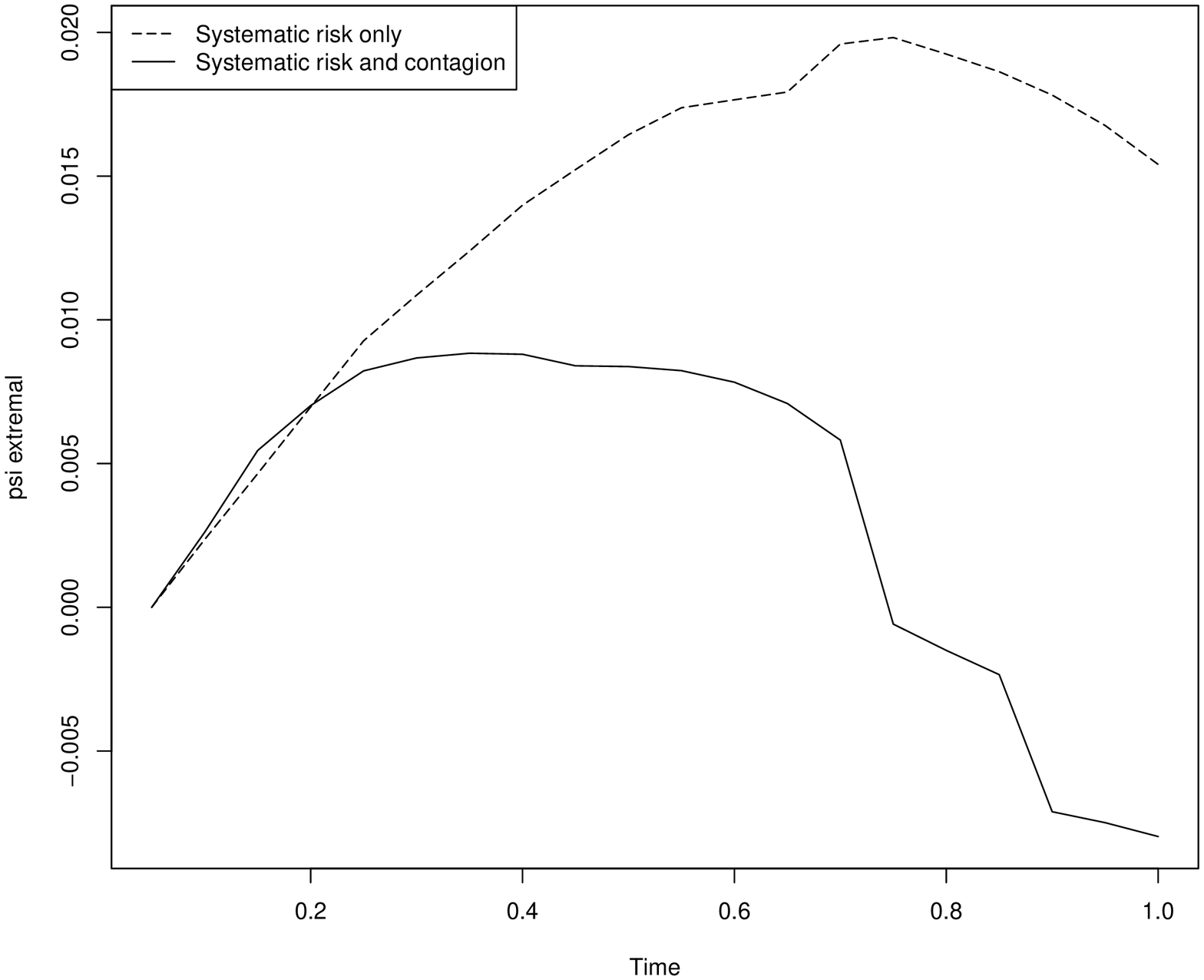}
\end{minipage}
\caption{\label{fig:phiHeter2} \emph{Left panel}: Comparison of $\phi$ extremals for types A and B when systematic risk effects are present. \emph{Right panel}:  Optimal $\psi(t)$ for $t\in[0,1]$ and $\ell=0.85$.}
\end{figure}

\section{Proof of the large deviations principle}\label{S:LDPheterogeneous}

The proof of Theorem \ref{T:LDPHeterogeneous1} goes in two steps. First we prove the corresponding result when $\beta^{C}=0$, namely we prove the LDP in the heterogeneous case when defaults are independent of each other (Subsection \ref{SS:LDPIndHeterogeneous}). Then, based on this result combined with Theorem \ref{T:TransferResult2}, we shall obtain the desired LDP of of Theorem \ref{T:LDPHeterogeneous1} in Subsection \ref{SS:LDPDependentHeterogeneous}.
In Subsection \ref{SS:LDPDependentHeterogeneous2}, we prove Theorem \ref{T:LDPHeterogeneous2}.

\subsection{Large deviations under independence and heterogeneity}\label{SS:LDPIndHeterogeneous}

For each $N\in \N$, let $\{\tau^{\ind,N}_n\}_{n=1}^N$ be the independent collection
of default times with intensities
\begin{equation} \label{E:mainIndependent}
\begin{aligned}
d\lambda^{\ind,\NN}_t &= -\alpha_\NN (\lambda^{\ind,\NN}_t-\bar \lambda_\NN)dt + \sigma_\NN \sqrt{\lambda^{\ind,\NN}_t}dW^n_t  \qquad t>0\\
\lambda^{\ind,\NN}_0 &= \lambda_{\circ, N,n}.
 \end{aligned}
\end{equation}
In other words,
\begin{equation*} \tau^{\ind,N}_n\Def \inf\lb t\ge 0:\, \int_{0}^t \lambda^{\ind,\NN}_s\ge \ee_n\rb \end{equation*}
where the $\ee_n$'s are as in \eqref{E:tau}.  Then $\{\tau^{\ind,N}_n\}_{n=1}^N$
is a collection of independent random variables (although they are not identically distributed ).  Define the empirical measure of the types and default times:
\[
\bnu^{\ind,N}=\frac{1}{N}\sum_{n=1}^{N}\delta_{{\pp}^{\NN},\tau^{\ind,N}_n}.
\]
Moreover, let
\begin{equation}
\bar{H}^{\ind}(\nu)=\begin{cases}
\int_{{\pp}\in{\Types}}H\left(\frac{d\nu}{dU}({\pp}),\mu^{{\pp}}_{0,0}\right)U(d{\pp}) & \text{ if $\frac{\partial\nu}{\partial U}$ exists}\\
\infty &\textrm{otherwise}
\end{cases}
\end{equation}
where $H$ is as in \eqref{E:ClassicalEntropy}.  The definition of $\bar H^{\ind}$ differs from the definition \eqref{E:HDef} of $\bar H$ in that we replace $\mu^{\pp}_{\bar \nu,0}$ in \eqref{E:HDef} by $\mu^{\pp}_{0,0}$.  This corresponds to no contagion or systematic effects, which is the same as
setting $\beta^C=\beta^S=0$ leading to independent default rates.

Then, the corresponding large deviations principle reads as follows.
\begin{theorem}\label{T:LDPmeasuresIndependentHeterogeneous}
The family $\{\bnu^{\ind,N}, N\in\N\}$ satisfies the large deviations principle with rate function $\bar{H}^{\textrm{ind}}(\nu):\Pspace({\Types}\times \TInt)\mapsto[0,\infty]$.
\end{theorem}
\begin{proof}
The upper and lower bound follow by Lemmas \ref{L:LDPUpperBoundClosed} and \ref{L:LDPindependentLowerBound} respectively. Let us now prove that this is a lower-semicontinuous functional with compact level sets. Lower-semicontinuity of $H\left(\frac{d\nu}{dU}({\pp}),\mu^{{\pp}}\right)$ follows immediately, since
$H(\nu,\mu)$ is a convex, lower-semincontinuous function of each variable $\nu$ or $\mu$ separately (e.g., Lemma 1.4.3 in \cite{DupuisEllis}). Compactness of level sets follows again as a
consequence of the corresponding property of the relative entropy $H(\nu,\mu)$ (e.g., Lemma 1.4.3 in \cite{DupuisEllis}). These properties are being inherited to $\bar{H}^{\textrm{ind}}(\nu)$, by Lemma 6.2.16 of \cite{DemboZeitouni}. This completes the proof of the Theorem.
\end{proof}

\begin{remark}
 Under homogeneity, namely if ${\pp}^{\NN}_{0}={\pp}_{0}$ for all $N\in\N$ and $n\in\{1,\cdots,N\}$ then the LDP is a direct consequence of Sanov's theorem and is given by
\eqref{E:ClassicalEntropy}. The fact that the defaults are not identically distributed posses some additional difficulties in the proof, as it is seen below.
\end{remark}

\subsubsection{Compact Support}\label{SSS:ExponentialTightness}
It is often useful (in passing from a large deviations upper bound for compact sets to an upper bound for closed sets) to show sufficient tightness.  In our case,
Assumption \ref{A:Bounded} actually gives us compact support.
Assumption \ref{A:Bounded} implies that there is a bounded subset $K$ of $ \Types$
such that $ \pp^\NN_0\in K$ for all $N\in \N$ and $n\in \{1,2\dots N\}$.
Since $\TInt$ is itself compact, $[-\KK_{\ref{A:Bounded}},\KK_{\ref{A:Bounded}}]^6\times \TInt$ is compact, and fact
\begin{equation*} \bnu^{\ind,N}\left([-\KK_{\ref{A:Bounded}},\KK_{\ref{A:Bounded}}]^6\times \TInt\right) = 1. \end{equation*}
Moving this to measure space, let's define
\begin{equation}\label{E:KDef} \KCal \Def \lb \omega\in \Pspace(\PT):\, \omega\left([-\KK_{\ref{A:Bounded}},\KK_{\ref{A:Bounded}}]^6\times \TInt\right) = 1\rb. \end{equation}
Then $\KCal\subset\subset \Pspace(\PT)$, and $\BP$-a.s. $\bnu^{\ind,N}\in \KCal$
for all $N\in \N$.

\subsubsection{Large deviations upper bound}\label{SSS:LDPIndependentUpperBound}
In this section, we prove the large deviations upper bound.

Let's start with an equivalent characterization of $\bar{H}(\nu)$.
\begin{lemma}\label{L:RateFunctionEquivalentCharacterization}
For every $\nu\in\Pspace(\PT)$ we have that
\[
\bar{H}(\nu)=\sup_{\phi\in C(\Types\times \TInt)}\left\{ \int_{(\pp,t)\in \PT}\phi(\pp,t)\nu(d\pp,dt)-\int_{\pp\in \Types}\log\int_{t\in \TInt}e^{\phi(\pp,t)}\mu^{\pp}_{0,0}(dt) U(d{\pp}) \right\}.
\]
\end{lemma}
\begin{proof}  First we assume that the stochastic kernel $\xi=\frac{\partial \nu}{\partial U}$ exists.  For $\phi\in C(\TInt)$ and $\pp\in \PP$, let's define
\begin{equation*} G(\phi,\pp) \Def \int_{t\in \TInt}\phi(t)\xi(\pp)(dt) - \log \int_{t\in \TInt}e^{\phi(t)}\mu^\pp_{0,0}(dt). \end{equation*}
The Donsker-Varadhan variational representation of the relative entropy \cite[Theorem 1.4.3]{DupuisEllis} is that
\begin{equation*} H(\xi(\pp),\mu^{\pp}_{0,0}) = \sup_{\phi\in C(\TInt)}G(\phi,\pp). \end{equation*}

For any $\phi\in C(\PT;\R)$, we have that
\begin{align*}
&\int_{(\pp,t)\in \Types\times \TInt}\phi(\pp,t)\nu(d\pp,dt)-\int_{\pp\in \Types}\log\int_{t\in \TInt}e^{\phi(\pp,t)}\mu^{\pp}_{0,0}(dt) U(d\pp) \\
&=\int_{\pp\in \Types}\lb \int_{t\in \TInt}\phi(\pp,t)\xi(\pp)(dt)-\log\int_{t\in \TInt}e^{\phi(\pp,t)}\mu^{\pp}_{0,0}(dt)\rb U(d\pp) \\
&=\int_{\pp\in \Types}G(\phi(\pp,\cdot),\pp) U(d\pp)
\le \int_{\pp\in \Types}H(\xi(\pp),\mu^\pp_{0,0}) U(d\pp)
= \bar H^{\ind}(\nu).
\end{align*}

To prove the opposite inequality, define
\[
F_N(\pp)=\sup_{\phi\in C(\PT), \|\phi\|\leq N}G(\phi(\pp,\cdot),\pp)
\]
for all positive integers $N$.  Fix now $\eta>0$.  For each $N$, let
$\phi_N^{*} \in C(\PT;\R)$ be such that $\|\phi_N^*\|\le N$
and $G(\phi^*_N(\pp,\cdot),\pp)\ge F_N(\pp)-\eta$.
Since $N\mapsto F_N(\pp)$ is non-decreasing for each $\pp\in \PP$,
monotone convergence implies that
\begin{align*}
\bar{H}(\nu)&=\int_{\pp\in \Types}\lim_{N\rightarrow\infty}F_N(\pp)U(d\pp)
\leq \lim_{N\rightarrow\infty} \int_{\pp\in \Types}F_N(\pp)U(d{\pp})\\
&\leq \lim_{N\rightarrow\infty} \int_{\pp\in \Types}G(\phi_N^{*}(\pp),\pp)U(d{\pp})+\eta
\leq \sup_{\phi\in C(\PT)} \int_{{\Types}}G(\phi,{\pp})U(d{\pp})+\eta\\
\end{align*}
Let $\eta\searrow 0$; combining things together, we have the claim when $\frac{\partial \nu}{\partial U}$ exists.

Let's next consider the case where $\frac{\partial \nu}{\partial U}$ is not well defined; thus $\bar H^{\ind}(\nu)=\infty$. Then there
is an $A\in \Borel(\Types)$ and $B\in \Borel(\TInt)$ such that
$\nu(A\times B)>0$ and $U(B)=0$.  Since $\PP$ is Polish,
there is a closed subset $F$ of $A$ such that $\nu(F\times B)>0$.
For $c>0$ and $N\in \N$, define
\[
\phi^c_N(\pp,t)=c \exp\left[-N\operatorname{dist}_{\PT}((\pp,t),F\times B)\right]
\]
where $\operatorname{dist}_{\PT}(\cdot,F\times B)$ is the distance (in $\PT$) to $F\times B$.  Note that
\begin{equation*} \lim_{N\to\infty}\phi^c_N = c\chi_{F\times B} \end{equation*}
pointwise.  By dominated convergence, we then have that
\begin{align*}
&\sup_{\phi\in C(\PT)}\lb \int_{(\pp,t)\in \PT}\phi(\pp,t)\nu(d\pp,dt)
-\int_{\pp\in \PP}\lb \log \int_{t\in \TInt}e^{\phi(\pp,t)}\mu^\pp_{0,0}(dt)\rb U(d\pp)\rb \\
&\qquad \ge \varlimsup_{N\to \infty}\lb \int_{(\pp,t)\in \PT}\phi^c_N(\pp,t)\nu(d\pp,dt)
-\int_{\pp\in \PP}\lb \log \int_{t\in \TInt}e^{\phi^c_N(\pp,t)}\mu^\pp_{0,0}(dt)\rb U(d\pp)\rb \\
&\qquad = \int_{(\pp,t)\in \PT}c\chi_{F\times B}(\pp,t)\nu(d\pp,dt)
-\int_{\pp\in \PP}\lb \log \int_{t\in \TInt}e^{c\chi_{F\times B}(\pp,t)}\mu^\pp_{0,0}(dt)\rb U(d\pp)
\end{align*}
If $\pp\in F^c$, then
\begin{equation*} \int_{t\in \TInt}e^{c\chi_{F\times B}(\pp,t)}\mu^\pp_{0,0}(dt) = \int_{t\in \TInt}1\mu^\pp_{0,0}(dt) = \mu^\pp_{0,0}(\TInt)=1. \end{equation*}
Thus
\begin{equation*}
\sup_{\phi\in C(\PT)}\lb \int_{(\pp,t)\in \PT}\phi(\pp,t)\nu(d\pp,dt)
-\int_{\pp\in \PP}\lb \log \int_{t\in \TInt}e^{\phi(\pp,t)}\mu^\pp_{0,0}(dt)\rb U(d\pp)\rb \ge c\nu(F\times B).
\end{equation*}
Letting $c\nearrow \infty$, we get the claim when $\frac{\partial \nu}{\partial U}$ is not well defined, finishing the proof.
\end{proof}

We now can prove the upper bound.
\begin{lemma}\label{L:LDPUpperBoundClosed}
For any closed set $F\subset \Pspace(\PT)$, we have
\[
\varlimsup_{N\rightarrow\infty}\frac{1}{N}\ln\BP\lb \bnu^{\ind,N}\in F\rb \leq -\inf_{\nu\in F}H^{\ind} (\nu)
\]
\end{lemma}
\begin{proof}  Using \eqref{E:KDef}, we start with the fact that
\begin{equation}\label{E:cutoff} \BP\lb\bnu^{\ind,N}\in F\rb = \BP\lb\bnu^{\ind,N}\in F\cap \KCal\rb. \end{equation}

Fix $\ell<\inf_{\nu\in F\cap \KCal}\bar H^{\ind}(\nu) $.   For every $\phi\in C(\PT;\R)$, define
\[
A_\phi=\left\{\nu\in \Pspace(\PT): \int_{(\pp,t)\in \PT}\phi(\pp,t)\nu(d\pp,dt)-\int_{\pp\in \Types}\lb \log \int_{t\in \TInt}e^{\phi(\pp,t)}\mu^{\pp}(dt)\rb U(d\pp)>\ell\right\}
\]
By Lemma \ref{L:RateFunctionEquivalentCharacterization}, we have that
\begin{equation*} F\cap \KCal \subset \cup_{\phi\in C(\PT)}A_\phi. \end{equation*}
Since $F\cap \KCal$ is compact, we in fact have that
\begin{equation*} F\cap \KCal \subset \cup_{\phi\in \Phi}A_\phi \end{equation*}
for some finite subset $\Phi$ of $C(\PT;\R)$.
Consequently
\begin{equation}\label{E:cover} \BP\lb\bnu^{\ind,N}\in F\cap \KCal\rb \leq \sum_{\phi\in \Phi}\BP\{\nu^{\ind,N}\in A_\phi\}. \end{equation}

Using the exponential Chebychev inequality, we calculate that
\begin{align*}
\BP\lb\bnu^{\ind,N}\in A_\phi\rb
&\leq \BP\lb \int_{(\pp,t)\in \PT}\phi(\pp,t)\bnu^{\ind,N}(d\pp,dt)\right.\\
&\qquad \left. >\ell+\int_{\pp\in \Types}\lb \log \int_{t\in \TInt}e^{\phi(\pp,t)}\mu^{\pp}_{0,0}(dt)\rb U(d\pp)\rb \\
&= \BP\lb N\int_{(\pp,t)\in \PT}\phi(\pp,t)\bnu^{\ind,N}(d\pp,dt)\right.\\
&\qquad \left.>N\ell+N\int_{\pp\in \Types}\lb \log \int_{t\in \TInt}e^{\phi(\pp,t)}\mu^{\pp}_{0,0}(dt)\rb U(d\pp)\rb \\
&\le \exp\left[-N\ell\right] \BE\exp\left[N\int_{(\pp,t)\in \PT}\phi(\pp,t)\bnu^{\ind,N}(d\pp,dt)\right]\\
&\qquad \times \exp\left[-N\int_{\pp\in \Types}\lb \log \int_{t\in \TInt}e^{\phi(\pp,t)}\mu^{\pp}_{0,0}(dt)\rb U(d\pp)\right]\\
&\le \exp\left[-N\ell\right] \exp\left[N\int_{\pp\in \Types}\lb \log \int_{t\in \TInt}e^{\phi(\pp,t)}\mu^{\pp}_{0,0}(dt)\rb U_N(d\pp)\right]\\
&\qquad \times \exp\left[-N\int_{\pp\in \Types}\lb \log \int_{t\in \TInt}e^{\phi(\pp,t)}\mu^{\pp}_{0,0}(dt)\rb U(d\pp)\right]
\end{align*}

Thus
\[
\varlimsup_{N\rightarrow\infty}\frac{1}{N}\ln\BP\lb \bnu^{\ind,N}\in A_\phi\rb\leq -\ell.
\]

Using this a finite number of times in \eqref{E:cover}, we get that
\begin{equation*} \varlimsup_{N\rightarrow\infty}\frac{1}{N}\ln\BP\lb\bnu^{\ind,N}\in F\cap \KCal\rb \le -\ell. \end{equation*}

Letting $\ell\nearrow \inf_{\nu\in F\cap \KCal}\bar H^{\ind}(\nu)$, we have that
\begin{equation*} \varlimsup_{N\rightarrow\infty}\frac{1}{N}\ln\BP\lb\bnu^{\ind,N}\in F\cap \KCal\rb \le -\inf_{\nu\in F\cap \KCal}\bar H^{\ind}(\nu). \end{equation*}

Finally returning to \eqref{E:cutoff}, we calculate that
\begin{equation*}\varlimsup_{N\rightarrow\infty}\frac{1}{N}\ln\BP\lb\bnu^{\ind,N}\in F\rb
\le   -\inf_{\nu\in F\cap \KCal}\bar H^{\ind}(\nu) \le -\inf_{\nu\in F}\bar H^{\ind}(\nu)
\end{equation*}
giving us the claim.
\end{proof}

\subsubsection{Large deviations lower bound}\label{SSS:LDPIndependentLowerBound}
In this section, we prove the large deviations lower bound.

We need the following continuity result.
\begin{lemma}\label{L:Continuity} The map $\pp\mapsto \mu^{\pp}_{0,0}$
is continuous as a map from $\PP$ to $\Pspace(\TInt)$.\end{lemma}
\begin{proof} This follows fairly easily from the explicit formula \eqref{E:muDef}.\end{proof}
\noindent We can now prove the lower bound.
\begin{lemma}\label{L:LDPindependentLowerBound}
Let $G$ be an open subset of $\Pspace(\PT)$. Then
\begin{equation*}
\varliminf_{N\rightarrow\infty}\frac{1}{N}\ln\BP\lb \bnu^{\ind,N}\in G\rb\geq -\inf_{\nu\in G}\bar H^{\ind}(\nu)
\end{equation*}
\end{lemma}
\begin{proof}  We proceed in a standard way.  It suffices to fix a $\nu^*\in G$ such that $\bar H^{\ind}(\nu^*)<\infty$ and a $\eta>0$ and show that
\begin{equation}\label{E:goal}
\varliminf_{N\rightarrow\infty}\frac{1}{N}\ln\BP\lb \bnu^{\ind,N}\in G\rb\geq -\bar H^{\ind}(\nu^*)-\eta.
\end{equation}

Let's understand the implications of $\bar H^{\ind}(\nu^*)<\infty$.  Then $\xi\Def \frac{\partial \nu*}{\partial U}$ exists and
\begin{equation*} \int_{\pp\in \PP}H(\xi(\pp),\mu^{\pp}_{0,0})U(d\pp)<\infty \end{equation*}
so $H(\xi(\pp),\mu^{\pp}_{0,0})<\infty$ for $U$-a.e. $\pp\in \PP$.  Thus
$\xi(\pp)\ll \mu^{\pp}_{0,0}$ for $U$-a.e. $\pp\in \PP$.  Finally, this and Tonelli's theorem ensures that
\begin{equation*} \phi(\pp,t) \Def \frac{d\xi(\pp)}{d\mu^{\pp}_{0,0}}(t) \end{equation*}
is well-defined.  In fact, defining the stochastic kernel $\mu^{\bullet}_{0,0}(p)\Def \mu^{\pp}_{0,0}$, we have that
\begin{equation}\label{E:RND} \phi = \frac{d\nu^*}{d(\mu^{\bullet}_{0,0}\otimes U)} \end{equation}
so $\phi$ is measurable.  Using \eqref{E:RND}, we furthermore note that
\begin{equation*} \nu^*\lb (\pp,t)\in \PT:\, \phi(\pp,t)=0\rb
= \int_{(\pp,t)\in \PT}\chi_{\{0\}}(\phi(\pp,t))\phi(\pp,t) (\mu^{\bullet}_{0,0}\otimes U)(d\pp,dt) = 0. \end{equation*}

Let's thus define
\begin{equation*} \psi(\pp,t) \Def \begin{cases} \ln \phi(\pp,t) &\text{if $\phi(\pp,t)>0$} \\
-\infty &\text{if $\phi(\pp,t)=0$} \end{cases}\end{equation*}

Defining, for convenience $e^{-\infty}\Def 0$; we know that
\begin{equation}\label{E:unity} \int_{t\in \TInt}e^{\psi(\pp,t)}\mu^\pp_{0,0}(dt)= \xi(\pp)(B)=1. \end{equation}
for $U$-a.e. $\pp\in \PP$.  We also note that
\begin{equation*} \bar H^\ind(\nu^*) = \int_{(\pp,t)\in \PT}\psi(\pp,t)\nu^*(d\pp,dt)-\int_{\pp\in \Types}\lb \log \int_{t\in \TInt}e^{\psi(\pp,t)}\mu^{\pp}_{0,0}(dt)\rb U(d\pp). \end{equation*}

Fix some $\mu^\circ\in \Pspace(\TInt)$.  For $\Psi\in C(\PT;\R)$, define now the stochastic kernel
\begin{equation*} \Xi_\Psi(\pp)(B)\Def \begin{cases} \frac{\int_{t\in B}\exp\left[\Psi(\pp,t)\right]\mu^\pp_{0,0}(dt)}{\int_{t\in \TInt}\exp\left[\Psi(\pp,t)\right]\mu^\pp_{0,0}(dt)} &\text{if $\int_{t\in \TInt}\exp\left[\Psi(\pp,t)\right]\mu^\pp_{0,0}(dt)>0$} \\
\mu^\circ(B) &\text{if $\int_{t\in \TInt}\exp\left[\Psi(\pp,t)\right]\mu^\pp_{0,0}(dt)=0$} \end{cases} \qquad B\in \Borel(\TInt) \end{equation*}

Let's now approximate, keeping in mind Lemma \ref{L:Continuity}, \eqref{E:RND}, and the fact that $x\mapsto e^{-x}$ is continuous on $[-\infty,\infty)$.  We then have that there is a $\hat \psi\in C(\PT;\R)$ such that
\begin{equation*} \left|\bar H^\ind(\nu^*) - \int_{(\pp,t)\in \PT}\hat \psi(\pp,t)\nu^*(d\pp,dt)-\int_{\pp\in \Types}\lb \log \int_{t\in \TInt}e^{\hat \psi(\pp,t)}\mu^{\pp}_{0,0}(dt)\rb U(d\pp)\right|<\eta/3 \end{equation*}
and such that the measure $\hat \nu\Def \Xi_{\hat \psi}\otimes U$ is in $G$.

Let's now define
\begin{equation}\label{E:SDef} S_N \Def \int_{(\pp,t)\in \PT}\hat \psi(\pp,t)\bnu^{\ind,N}(d\pp,dt)-\int_{\pp\in \Types}\lb \log \int_{t\in \TInt}e^{\hat \psi(\pp,t)}\mu^{\pp}_{0,0}(dt)\rb U_N(d\pp). \end{equation}
Then $\BE\left[\exp\left[NS_N\right]\right]=1$, and we thus define
\begin{equation*} \hat \BP_N(A)\Def \BE\left[\chi_A \exp\left[N S_N\right]\right]\qquad A\in \filt \end{equation*}

The next step is to reduce $G$.  Let $B\subset \PT$ be an open subset of $G$
which contains $\hat \nu$ and such that
\begin{equation*} \left|\int_{(\pp,t)\in \PT}\hat \psi(\pp,t)\nu(d\pp,dt)-\int_{(\pp,t)\in \PT} \hat \psi(\pp,t)\nu^*(d\pp,dt)\right|<\eta/2 \end{equation*}
for all $\nu\in B$.  Thus
\begin{equation*} \BP\{\bnu^{\ind,N}\in G\} \ge \BP\{\bnu^{\ind,N}\in B\}
= \hat \BE_N\left[\chi_{\{\bnu^{\ind,N}\in B\}}\exp\left[-NS_N\right]\right]. \end{equation*}

Thanks to Lemma \ref{L:Continuity}, the map
\begin{equation*} \pp\mapsto \int_{t\in \TInt}e^{\hat \psi(\pp,t)}\mu^{\pp}_{0,0}(dt) \end{equation*}
is in $C(\PP;\R)$.  Thus we can find an $N^*\in \N$ such that
\begin{equation*} \left| \int_{\pp\in \Types}\lb \log \int_{t\in \TInt}e^{\hat \psi(\pp,t)}\mu^{\pp}_{0,0}(dt)\rb U(d\pp) - \int_{\pp\in \Types}\lb \log \int_{t\in \TInt}e^{\hat \psi(\pp,t)}\mu^{\pp}_{0,0}(dt)\rb U_N(d\pp)\right|\le \eta/2 \end{equation*}
if $N\ge N^*$.  Combining things together, we have that
\begin{equation*} \left|\int_{(\pp,t)\in \PT}\hat \psi(\pp,t)\bnu^{\ind,N}(d\pp,dt)-\int_{\pp\in \Types}\lb \log \int_{t\in \TInt}e^{\hat \psi(\pp,t)}\mu^{\pp}_{0,0}(dt)\rb U_N(d\pp)-\bar H^\ind(\nu^{*})\right|<\eta
\end{equation*}
if $\bnu^{\ind,N}\in B$ and $N\ge N^*$.

Assume now that $N\ge N^*$.  Then
\begin{align*} \BP\{\bnu^{\ind,N}\in G\} &\ge \BP\{\bnu^{\ind,N}\in B\}
= \hat \BE_N\left[\chi_{\{\bnu^{\ind,N}\in B\}}\exp\left[-NS_N\right]\right]\\
&\ge \hat \BP_N\{\bnu^{\ind,N}\in B\}\exp\left[-N\{\bar H^\ind(\nu^*)+\eta\}\right]. \end{align*}

To finish the proof, let's show that
\begin{equation}\label{E:lowergoal} \varliminf_{N\to \infty}\hat \BP_N\{\bnu^{\ind,N}\in B\}>0. \end{equation}

To do so, let's construct a metric on $\Pspace(\PT)$.  Define $d(x) \Def |x|/(1+|x|)$;
then the map $(x,y)\mapsto d(x-y)$ is a metric on $\R$.  Let $\{f_n\}_{n\in\N}$
be a countable and dense subset of $C(\PT;\R)$, and define
\begin{equation}\label{E:metric} \rho(\nu_1,\nu_2) \Def \sum_{n\in \N}2^{-n}d\left(\int_{(\pp,t)\in \PT}f_n(\pp,t)\nu_1(d\pp,dt)-\int_{(\pp,t)\in \PT}f_n(\pp,t)\nu_2(d\pp,dt)\right) \end{equation}
for all $\nu_1$ and $\nu_2$ in $\Pspace(\PT)$.  Then $\rho$ is a metric on $\Pspace(\PT)$.  Since $\hat \nu\in B$, there is a $\delta_{\eqref{E:inclusion}}>0$ small enough such that
\begin{equation}\label{E:inclusion} \lb \nu'\in \Pspace(\PT):\, \rho(\nu',\hat \nu)<\delta_{\eqref{E:inclusion}}\rb \subset B. \end{equation}

Let's now understand the statistics of $\nu^{\ind,N}$ under $\hat \BP_N$.
From the structure \eqref{E:SDef} of $S_N$, we see that, under $\hat \BP_N$,
$\{\tau^{\ind,N}_n\}_{n=1}^N$ are independent and $\tau^N_n$ has law $\Xi_{\hat \psi}(\pp^N_n)$.

Fix $f\in C(\PT;\R)$.  Let's write
\begin{align*} &\left|\int_{(\pp,t)\in \PT}f(\pp,t)\bnu^{\ind,N}(d\pp,dt)-\int_{(\pp,t)\in \PT}f(\pp,t)\hat \nu(d\pp,dt)\right| \\
&\qquad \le \left|\int_{(\pp,t)\in \PT}f(\pp,t)\bnu^{\ind,N}(d\pp,dt)-\int_{(\pp,t)\in \PT}f(\pp,t)(\Xi_{\hat \psi}\otimes U_N)(d\pp,dt)\right|\\
&\qquad \qquad + \left|\int_{(\pp,t)\in \PT}f(\pp,t)(\Xi_{\hat \psi}\otimes U_N)(d\pp,dt)-\int_{(\pp,t)\in \PT}f(\pp,t)(\Xi_{\hat \psi}\otimes U)(d\pp,dt)\right|. \end{align*}
Keeping Lemma \ref{L:Continuity} in mind, we have that
\begin{equation*} \lim_{N\to \infty}\left|\int_{(\pp,t)\in \PT}f(\pp,t)(\Xi_{\hat \psi}\otimes U_N)(d\pp,dt)-\int_{(\pp,t)\in \PT}f(\pp,t)(\Xi_{\hat \psi}\otimes U)(d\pp,dt)\right|=0. \end{equation*}
Let's next use Chebychev's inequality; we have arranged things to take
advantage of the fact that
\begin{equation*} \hat \BE_N\left[f(\pp^N_n,\sigma^N_n)\right] = \int_{t\in \TInt}f(\pp^N_n,t)\Xi_{\hat \psi}^{\pp^N_n}(dt) \end{equation*}
for all $N$ and $n$.
We then calculate that
\begin{align*} &\hat \BE_N\left[\left|\int_{(\pp,t)\in \PT}f(\pp,t)\bnu^{\ind,N}(d\pp,dt)-\int_{(\pp,t)\in \PT}f(\pp,t)(\Xi_{\hat \psi}\otimes U_N)(d\pp,dt)\right|^2\right]\\
&\qquad = \hat \BE_N\left[\left|\frac1N\sum_{n=1}^N \lb f(\pp^N_n,\sigma^N_n)-\int_{t\in \TInt}f(\pp^N_n,t)\Xi_{\hat \psi}(\pp^N_n)(dt)\rb\right|^2\right] \\
&\qquad = \frac{1}{N^2}\sum_{n=1}^N \hat \BE_N\left[\lb f(\pp^N_n,\sigma^N_n)-\int_{t\in \TInt}f(\pp^N_n,t)\Xi_{\hat \psi}(\pp^N_n)(dt)\rb^2\right] \\
&\qquad \le \frac{4}{N}\sup_{(\pp,t)\in \PT}|f(\pp,t)|^2.
\end{align*}

Let's collect things together.  For each $f\in C(\PT;\R)$, we have that
\begin{equation*} \lim_{N\to \infty}\hat \BE_N\left[\left|\int_{(\pp,t)\in \PT}f(\pp,t)\bnu^{\ind,N}(d\pp,dt)-\int_{(\pp,t)\in \PT}f(\pp,t)\hat \nu(d\pp,dt)\right|\right]=0. \end{equation*}
The structure of \eqref{E:metric} then implies that
\begin{equation*} \lim_{N\to \infty}\hat \BP_N\lb \rho(\bnu^{\ind,N},\hat \nu)>\delta_{\eqref{E:inclusion}}\rb =0. \end{equation*}
This then implies \eqref{E:lowergoal}, finishing the proof.
\end{proof}

\subsection{Proof of Theorem \ref{T:LDPHeterogeneous1}}\label{SS:LDPDependentHeterogeneous}

In this section we prove Theorem \ref{T:LDPHeterogeneous1}, using Varadhan's integral lemma \cite{Varadhan1966} as it appears in Theorem II.7.2 of \cite{Ellis}.
\begin{theorem}[Varadhan's integral lemma]\label{T:TransferResult2}
Let $S$ be a Polish space.  Suppose that $\{\xi^N\}_{N\in\N}$ is a sequence of $S$-valued random variables with large deviations principle with rate function $I$.
Let $\{\tilde \xi^N\}_{N\in \N}$ be another sequence of $S$-valued random variables and assume that there is a continuous function $G:S\to \R$ such that
\begin{equation*} \BP\{\tilde \xi_N\in A\} = \BE\left[\chi_A(\xi_N)\exp\left[NG(\xi_N)\right]\right] \end{equation*}
for all $A\in \Borel(S)$.  If
\begin{equation*} \lim_{\alpha\rightarrow\infty}\varlimsup_{N\rightarrow\infty}\frac{1}{N}\ln \BE\left[\chi_{[\alpha,\infty)}\left(G(\xi^N)\right)e^{N G(\xi^{N})}\right]=-\infty
\end{equation*}
then $\{\tilde \xi^N\}_{N\in \N}$ has a large deviation principle with good rate function $I-G$.
\end{theorem}
\noindent
In our case, we take $S\Def \Pspace(\PT)$ and $\xi_N\Def \bnu^{\ind,N}$.
Lemma \ref{L:DensityFcn} allows us to define the Radon-Nikodym derivative $e^{NG}$.

For $\pp\in \Types$, define
\begin{align*}
\dot b^\pp(t) &= 1-\frac12 \sigma^2 \left(b^\pp(t)\right)^2 - \alpha b^\pp(t) \qquad t>0\\
b^\pp(0)&=0. \end{align*}

Note that since $\beta^S=0$ then  $\theta^{{\pp}}_{t}(s)=b^{{\pp}}(s)$ for all $0\leq s\leq t\leq T$, where $\theta^{{\pp}}_{t}(s)$ is defined in \eqref{E:DuffiePanSingleton2}.  It is easy to see that $b^{\pp}$ and $\dot b^{\pp}$ are uniformly bounded in $\Types\times [0,T]$.  Using \eqref{Eq:ProbDistribution} and Lemma \ref{L:DensityFcn} we can fairly easily compute the ratio of densities of $\tau^N_n$
and $\tau^{\ind,N}_n$ and consequently of $\bnu^N$ and $\bnu^{\ind,N}$.  For $\pp\in \Types$, $r\in [0,T]$, and $\nu\in \Pspace(\Types\times [0,T])$, define
\begin{align*}
g_{\nu}(\pp,r)\Def \log\frac{d\mu^{{\pp}}_{\bar{\nu},0}([0,r])}{d\mu^{{\pp}}_{0,0}([0,r])}&=\log\lb \dot b^{\pp}(r)\lambda_\circ+\alpha \bar \lambda b^\pp(r)
+\beta^C\int_{0}^{r}\dot b^{\pp}(r-u)\nu(\PP,du)\right)\\
&\qquad -\log\left(\dot b^{\pp}(r)\lambda_\circ+\alpha \bar \lambda b^\pp(r)\right)-\beta^C\int_{0}^{r}b^{\pp}(r-u)\nu(\PP,du)\end{align*}
For $\pp\in \Types$, let's also define
\begin{align*}
g_{\nu}(\pp,\pt) &= \beta^C\int_{0}^{T}b^{\pp}(T-u)\nu(\PP,du)\end{align*}

For $\nu\in \Pspace(\PT)$, let's finally define
\begin{equation*} G(\nu) \Def \int_{(\pp,t)\in \PT}g_{\nu}(\pp,t)\nu(d\pp,dt).\end{equation*}
Then
\begin{equation*} \BP\lb \bnu^N\in A\rb
= \BE\left[\chi_A(\bnu^{\ind,N})\exp\left[NG(\bnu^{\ind,N})\right]\right].\end{equation*}

Theorem \ref{T:LDPHeterogeneous1} follows by Theorem \ref{T:LDPmeasuresHeter}  and contraction principle.
\begin{theorem}\label{T:LDPmeasuresHeter}
The family $(\bnu^N)_{N\in\N}$ of \eqref{E:nudef} satisfies the large deviations principle with rate function $\bar{H}(\nu):\Pspace({\Types}\times \TInt) \mapsto[0,\infty]$.
\end{theorem}
\begin{proof}
Note that $G:\Pspace(\PT)\to \R$ is continuous in the weak topology and bounded.  Thus
\begin{equation*}
\lim_{\alpha\rightarrow\infty}\varlimsup_{N\rightarrow\infty}\frac{1}{N}\log\BE\left[\chi_{[\alpha,\infty)}\left(G(\bnu^{\ind,N})\right)\exp\left[N G(\bnu^{\ind,N})\right]\right]=-\infty
\end{equation*}
holds.  Thus $\{\bnu^N, N\in\N\}$ has a large deviations principle with rate function
$\bar H\Def \bar H^{\ind}-G$.  Fix $\nu\in \Pspace(\PT)$.
If $\frac{\partial \nu}{\partial U}$ does not exist, then $\bar H^{\ind}(\nu)=\infty$ so $\bar H(\nu)=\infty$.
If $\frac{\partial \nu}{\partial U}$ exists, then
\begin{align*}
\bar H (\nu)
&=\int_{\pp\in \Types}\lb \int_{t\in \TInt}\ln \frac{\tfrac{\partial \nu}{\partial U}(\pp)}{d \mu^{\pp}_{0,0}}(t) \frac{d\nu}{dU}(\pp, dt)\rb U(d\pp)-\int_{(\pp,t)\in \PT}\ln\frac{d\mu^{\pp}_{\bar \nu,0}(s)}{d\mu^{\pp}_{0,0}(s)}\nu(d\pp,ds)\\
&=\int_{\pp\in \Types}\lb \int_{t\in \TInt}\lb \ln \frac{\tfrac{\partial \nu}{\partial U}(\pp)}{d \mu^{\pp}_{0,0}}(t) - \ln\frac{d\mu^{\pp}_{\bar \nu,0}(s)}{d\mu^{\pp}_{0,0}(s)}\rb \frac{\partial \nu}{\partial U}(\pp,ds)\rb U(d\pp)\\
&=\int_{\pp\in \Types}\lb \int_{t\in \TInt}\ln \frac{\tfrac{\partial \nu}{\partial U}(\pp)}{d \mu^{\pp}_{\nu,0}}(t) \frac{d\nu}{dU}(\pp, dt)\rb U(d\pp)\\
&= \int_{\pp \in \Types}H\left(\frac{d\nu}{dU}(\pp),\mu^{\pp}_{\bar \nu,0}\right)U(d\pp)
\end{align*}
\end{proof}

We conclude with the proof of Corollary \ref{C:LDPgivingLLNHeter}.
\begin{proof}[Proof of Corollary \ref{C:LDPgivingLLNHeter}]
Clearly, it is enough to consider the case $\bar{H}(\nu)<\infty$.  $H\left(\cdot,\cdot\right)$ is relative entropy, which implies that it is a convex and lower semicontiunuous function with respect to both arguments and strictly convex as a function of the first argument for each fixed second argument. Thus $\bar{H}(\nu)=0$ implies that  $H\left(\frac{d \nu}{d U}({\pp}),\mu_{\bar{\nu},0}^{{\pp}}\right)=0$, which  due to the relative entropy nature of $H$, it is true if  $\frac{\partial \nu}{\partial U}({\pp})=\mu_{\bar{\nu},0}^{{\pp}}$ $U-$a.e. This implies the relation
\[
\nu({\Types}, B)=\int_{{\Types}\times B} \mu^{{\pp}}_{\bar{\nu},0}(ds) U(d{\pp}),\textrm{ for all } B\in\mathcal{B}(\TInt)
\]
which by \eqref{Eq:ProbDistribution} and Lemma \ref{L:DensityFcn}, translates to the equation
\begin{align}
\bar{\nu}([0,t])&=\int_{{\Types}}\left[1-\BE\left[e^{-\int_{0}^{t}\lambda^{\bar{\nu},0}_{s}({\pp})ds}\right]\right]U(d{\pp})\nonumber\\
&=1-\int_{{\Types}}e^{-\left(b^{{\pp}}(t)\lambda_{0}+\int_{0}^{t}\alpha\bar{\lambda}b^{{\pp}}(t-u)du+\beta^{C}\int_{0}^{t}b^{{\pp}}(t-u)\bar{\nu}(du)\right)}U(d{\pp})\label{Eq:LLNequation}
\end{align}
where $\lambda^{\nu,0}({\pp})$ satisfies \eqref{E:DuffiePanSingletonequation} with $\varphi(t)=\bar{\nu}([0,t])$ and $\psi(t)=0$.  Let us next argue that this equation has a unique solution which coincides with $L_{t}$, the limit in probability of the empirical loss $L^{N}_{t}$ as $N\rightarrow\infty$.

By Lemma 4.1 of \cite{GSS2013} we get that
\begin{equation}
L_{t}=1-\bmu_t( \PP)=1-\int_{\PP}\BE\left[e^{-\int_{0}^{t}\lambda^{\ast}_{s}(\pp)ds}\right]U(d\pp)\label{Eq:LLN}
\end{equation}
where for each $\pp\in\PP$, there is a unique pair
$\{(Q(t),\lambda_{t}^{*}( \pp)):t\in[0,T]\}$ taking values in
$\R_+\times\R_{+}$ such that
\begin{equation*}\label{E:bQDef} \begin{aligned} Q(t) &= \int_{\PP} \BE\left\{\lambda^{*}_{t}(\pp)\exp\left[-\int_{0}^t \lambda_s^*( \pp)ds\right]
\right\}U(d\pp).
\end{aligned}\end{equation*} and
\begin{equation*}
\lambda^*_t(\pp) = \lambda_\circ - \alpha\int_{0}^t (\lambda^*_s(\pp)-\bar \lambda)ds + \sigma\int_{0}^t\sqrt{\lambda^*_s(\pp)}dW^*_s +\beta^{C} \int_{0}^t Q(s) ds. \label{E:EffectiveEquation1}
\end{equation*}
Then, using Lemma \ref{L:DensityFcn}, one easily computes from (\ref{Eq:LLN}) that $L_{t}=\bar{\nu}([0,t])$, the unique solution of (\ref{Eq:LLNequation}).
\end{proof}

\subsection{Proof of Theorem \ref{T:LDPHeterogeneous2}}\label{SS:LDPDependentHeterogeneous2} In this Subsection, we consider Theorem \ref{T:LDPHeterogeneous2}.
The proof is a direct consequence of (a): Theorem \ref{T:LDPmeasuresHeter} and (b): the LDP for the $X^{N}_{t}=\eps_{N}X_{t}$ process, Lemma \ref{L:LDP_X}.
Thus, we only sketch the main arguments.

Firstly, we notice that an immediate consequence of Theorem \ref{T:LDPmeasuresHeter}, is the following conditional LDP.

\begin{lemma}\label{L:LDP1}
Consider the system defined in \eqref{E:main} with $\eps_{N}=1$. Under Assumptions \ref{A:Bounded} and \ref{A:regularity}, and given a path $t\mapsto X_{t}$ in $C_{c}\left([0,T];\R\right)$, the family $\{\bnu^{N},N\in\mathbb{N}\}$ satisfies the conditional large deviation principle with rate function $\bar{H}(\nu;\PT):\Pspace({\Types}\times \TInt)\mapsto[0,\infty]$ and speed $N$, where

\[
\bar{H}(\nu;\PT)=\begin{cases}
\int_{{\Types}}H\left(\frac{\partial\nu}{\partial U}({\pp}),\mu^{{\pp}}_{\bar{\nu},\PT}\right)U(d{\pp}), & \text{ if $\frac{\partial\nu}{\partial U}({\pp})$ exists}\\
+\infty, & otherwise.
\end{cases}
\]
\end{lemma}

Secondly, we notice that under Assumptions \ref{A:CoefficientsX}, \ref{A:CoefficientsX2} and \ref{A:CoefficientsX3} the family $\{X^{N}=\eps_{N}X, N\in\N\}$ satisfies a large deviations principle as given by Lemma \ref{L:LDP_X}.
\begin{lemma}\label{L:LDP_X}
Let $X^{N}=\eps_{N}X$ and assume that Assumptions \ref{A:CoefficientsX}, \ref{A:CoefficientsX2} and \ref{A:CoefficientsX3} hold.
 Then, the process $\{X^{N}_{\cdot}, N\in\N\}$ satisfies the large deviations principle with rate function $J_{X}(\cdot)$ as given by (\ref{Eq:RateFunctionXprocesses}) and speed $1/\eps_{N}^{2\zeta}$.
\end{lemma}
\begin{proof}
Notice that $X^{N}=\eps_{N}X$ is the unique strong solution of the SDE
\[
dX^{N}_{t}=\eps_{N}b\left(\frac{X^{N}_{t}}{\eps_{N}}\right) dt+\eps_{N}^{\zeta}\left[\eps_{N}^{1-\zeta}\kappa\left(\frac{X^{N}_{t}}{\eps_{N}}\right)\right]d V_{t},\quad X^{N}_{0}=\eps_{N}x_{0}
\]

It follows then by Theorem 2.1 in \cite{ChiariniFisher2012b}, or in the case of non-degenerate bounded diffusion from the classical results in Section 5.3. of \cite{FWBook}, that under Assumptions \ref{A:CoefficientsX}, \ref{A:CoefficientsX2} and \ref{A:CoefficientsX3} the processes $\{\eps_{N}X, N<\infty\}$ and $\{\bar{X}^{N}, N<\infty\}$ defined in (\ref{Eq:EquivalentAlternativeXprocess}) have the same large deviations principle  with rate function $J_{X}(\cdot)$ and where speed of the large deviations asymptotic is $1/\eps^{2\zeta}_{N}$.
\end{proof}

By Lemma \ref{L:LDP1} and contraction principle we get that, given  a path $t\mapsto X_{t}$, the large deviations rate function for the family $\{L^{N}_{t}=\bnu^{N}\left({\PP}\times[0,t]\right),N\in\N, t\in[0,T]\}$ is given by
\begin{equation}\label{Eq:ConditonalRateFcn}
J(\varphi;\PT)=\begin{cases}
\int_{{\Types}}g\left(\varphi({\pp}),f^{{\pp}}_{\bar{\varphi},\PT}(t)\right)
U(d{\pp}), & \varphi\in AC(\PP\times[0,T];\R),\varphi({\pp},0)=0,\varphi\geq 0,\\
&\qquad\text{and }\bar{\varphi}(s)=\int_{{\Types}}\varphi({\pp},s)U(d{\pp})\\
+\infty, & otherwise
\end{cases}
\end{equation}

The latter statement and Lemma \ref{L:LDP_X} imply  that the rate function for the pair $\left\{(L^{N}_{t},X^{N}_{t}), N\in\N, t\in[0,T]\right\}$ is given by $S(\varphi,\psi)$,
as defined in Theorem \ref{T:LDPHeterogeneous2}. Indeed,
by Lemma 1.2.18 and Theorem 4.1.11 of \cite{DemboZeitouni}, it is enough to
establish a local large deviations principle and exponential tightness. Exponential tightness is a direct consequence of the compact support computations of Section \ref{SSS:ExponentialTightness}
and the exponential tightness of the small noise diffusion process $X^{N}$. Denoting by $B(\cdot,\delta)$ the ball in the space of continuous functions in $[0,T]$ with uniform norm,
proving a local large deviations principle amounts to proving
\begin{eqnarray*}
 \lim_{\delta\downarrow 0}\limsup_{N\rightarrow\infty} \frac{1}{N}\log\BP \left\{(L^{N},X^{N})\in B((\varphi,\psi),\delta)\right\}&\leq& -S(\varphi,\psi), \quad \textrm{and}\nonumber\\
 \lim_{\delta\downarrow 0}\liminf_{N\rightarrow\infty} \frac{1}{N}\log\BP \left\{(L^{N},X^{N})\in B((\varphi,\psi),\delta)\right\}&\geq& -S(\varphi,\psi)
\end{eqnarray*}

In the case $\lim_{N\rightarrow\infty}N\eps^{2\zeta}_{N}= c\in(0,\infty)$, both statements are seen to be true due to the conditional probability relation $\BP(A\cap \Gamma)=\BP(A|\Gamma)\BP(\Gamma)$ and using the
conditonal large deviations principle discussed at Lemma \ref{L:LDP1} and the large deviations principle of Lemma \ref{L:LDP_X}. Basically, to exponential order, we  have as $N\rightarrow\infty$
\begin{equation}\label{E:genLDP}
 \BP\{L^N\approx \varphi,\, X^N \approx \psi\} \approx \exp\left[-N J(\varphi;\psi) - \frac{1}{\eps_N^{2\zeta}}J_{X}(\psi)\right].
 \end{equation}
which means that if $\lim_{N\rightarrow\infty}N\eps^{2\zeta}_{N}= c\in(0,\infty)$,  the rate function for the pair  $\left\{(L^{N}_{t},X^{N}_{t}), N\in\N, t\in[0,T]\right\}$
 is given by $S(\varphi,\psi)$, as defined in Theorem \ref{T:LDPHeterogeneous2}. Then, by varying over $\psi\in C([0,T];\R)$ and $\varphi\in C(\PP\times [0,T];\R)$ such that $\bar{\varphi}(T)=\ell$,
we get that the rate function for the loss  $\{L^{N}_{T},N<\infty\}$ at time $T$ is $I(\cdot)$, which is the statement of Theorem \ref{T:LDPHeterogeneous2}, concluding its proof.

Display (\ref{E:genLDP}) also shows that if  $\lim_{N\rightarrow\infty}N\eps^{2\zeta}_{N}=\infty$, then  the term $J(\varphi;\psi)$
is the dominant factor, while if $\lim_{N\rightarrow\infty}N\eps^{2\zeta}_{N}=0$, then the $J_{X}(\psi)$ entropy term is the dominant factor.

\section{Conclusion}\label{S:Conclusion}
In this paper we studied large deviations for an empirically motivated interacting particle system of default clustering. The components in the system interact through the empirical default
rate in the pool and through a systematic  risk which is common to all of the them. An explicit large deviations principle is derived and probabilities of tail events are then studied.
One can compute the extremals of the rate function, which characterize the most likely path to failure of the system. The numerical experiments reveal the role that contagion and
systematic risk play
in the failure of the system. In the numerical examples that we performed, we saw  that if a large default cluster occurs, the systematic risk is most likely to play a large role in the initial phase,
 but then its importance decreases, and then contagion effect become more important. Then, the analysis of the extremals in heterogeneous pools can help understanding which components in the pool are more vulnerable to the effect of contagion.


\section{Acknowledgements}
During preparation  of this work, K.S. was partially supported by the National Science Foundation (DMS 1312124) and R.S.  was partially supported by the National Science Foundation (DMS 1312071).

\end{document}